\documentclass[11pt,reqno]{amsart}
\usepackage{amsmath,amssymb,mathrsfs}
\usepackage{graphicx,cite}
\usepackage{color}
\usepackage{subfigure}
\usepackage{appendix}

\newtheorem{theorem}{Theorem}[section]

\newtheorem{lemma}[theorem]{Lemma}

\newtheorem{remark}[theorem]{Remark}

\newtheorem{assumption}{Assumption}
\newtheorem{definition}{Definition}

\newcommand{\ii}{{\rm i}}
\newcommand{\E}{\mathbb{E}}
\numberwithin{equation}{section}
\newcommand{\m}[1]{\boldsymbol{#1}}
 
\newcommand{\Q}{\mathcal{Q}}

\newcommand{\lt}{\left}
\newcommand{\rt}{\right}

\newcommand{\Cd}{\beta_d}

\DeclareMathOperator{\Tr}{tr}
\allowdisplaybreaks[4]

\begin{document}

\title[Stability estimates of inverse random source problems]{Stability estimates of inverse random source problems for the wave equations by using correlation-based data}

\author{Peijun Li}
\address{LSEC, ICMSEC, Academy of Mathematics and Systems Science, Chinese Academy of Sciences, Beijing 100190, China, and School of Mathematical Sciences, University of Chinese Academy of Sciences, Beijing 100049, China}
\email{lipeijun@lsec.cc.ac.cn}

\author{Ying Liang}
\address{Department of Mathematics, Duke University, Durham, NC 27708, USA}
\email{ying.liang@duke.edu}

\author{Xu Wang}
\address{LSEC, ICMSEC, Academy of Mathematics and Systems Science, Chinese Academy of Sciences, Beijing 100190, China, and School of Mathematical Sciences, University of Chinese Academy of Sciences, Beijing 100049, China}
\email{wangxu@lsec.cc.ac.cn}

\thanks{The third author is supported by NNSF of China (11971470, 11871068, and 12288201) and by CAS Project for Young Scientists in Basic Research (YSBR-087).}



\begin{abstract}
This paper focuses on stability estimates of the inverse random source problems for the polyharmonic, electromagnetic, and elastic wave equations. The source is represented as a microlocally isotropic Gaussian random field, which is defined by its covariance operator in the form of a classical pseudo-differential operator. The inverse problem is to  determine the strength function of the principal symbol by exploiting the correlation of far-field patterns associated with the stochastic wave equations at a single frequency. For the first time, we show in a unified framework that the optimal Lipschitz-type stability can be attained across all the considered wave equations through the utilization of correlation-based data.
\end{abstract}

\maketitle

\section{Introduction}
Scattering problems, arising from interactions between waves and media, hold significant importance across a broad spectrum of scientific domains, including medical imaging, exploration geophysics, remote sensing, and nondestructive testing \cite{ABG15, CK19}. The inverse source problem constitutes a crucial research subject in inverse scattering theory, which involves recovering unknown sources from wave field measurements. This problem is challenging due to the presence of non-radiating source, which results in non-uniqueness of solutions, particularly when using solely boundary measurements of wave fields at a single frequency \cite{AM-IP06, bleistein1977nonuniqueness, devaney1982nonuniqueness}. In \cite{BLT2010jde}, the authors initiated a study on employing multi-frequency data to ensure uniqueness and stability in the inverse source problem of the acoustic wave equation. Since then, there has been extensive research into enhancing the stability for the inverse source problem of various wave equations \cite{bao2020stability, cheng2016increasing, EI-SIMA20}. We refer to \cite{BLLT15} for a comprehensive review of theoretical and computational studies on solving inverse scattering problems with multiple frequency data.

The complexity of inverse scattering problems increases substantially with the introduction of random parameters into mathematical models \cite{FGPS-07}. These parameters are incorporated to address unpredictable environmental conditions, incomplete system information, and uncertainties stemming from measurement noise \cite{BPTB-IP02}. The introduction of randomness and uncertainties transforms deterministic inverse problems into demanding stochastic inverse problems. Consequently, statistical properties such as the mean and variance of these random parameters become essential for quantifying uncertainties associated with media or scatterers. The inverse random source problem was initially investigated in \cite{devaney1979jmp}, where several specific instances were examined to determine the auto-correlation of random sources. Further developments include a computational framework introduced in \cite{BCL16}, where random sources are represented as white noise, one of the most frequently adopted uncorrelated stochastic processes. 

To account for more general stochastic processes in characterizing uncertainty of underlying phenomena, a new model is introduced for inverse random source problems, inspired by \cite{LPS08}. The source is modeled as a generalized microlocally Gaussian (GMIG) random field, featuring a covariance operator that takes the form of a classical pseudo-differential operator. The inverse source problem involves determining the strength function of the principal symbol of the covariance operator through the utilization of statistics derived from the random wave field. The uniqueness of the inverse random source problems has been investigated for various wave equations, as exemplified in \cite{LW21}, demonstrating that the strength function is uniquely determined by averaging the amplitude of the random wave field across the frequency band, generated from a single realization of the random source. The issue of uniqueness in inverse random source problems has received considerable attention in previous research. However, studies addressing their stability are relatively rare. In \cite{LL23stability}, a recent study examined the stability of inverse source problems related to the three-dimensional Helmholtz equation, with the source modeled by white noise. Based on the correlation of near-field random wave fields, the H\"{o}lder- and logarithmic-type stability are established for homogeneous and inhomogeneous media, respectively. We reference \cite{HLOS-JMPA18} for the inverse problem concerning the determination of the metric tensor of the stochastic wave equation through the utilization of empirical correlations of the wave field. Related stability estimates regarding the inverse medium scattering problem in a deterministic setting can be found in \cite{hahner2001new, isaev2013new}. 

The objective of this work is to establish a unified framework for assessing the stability of inverse random source problems across the wave equations most commonly encountered. Specifically, we consider the stochastic polyharmonic, electromagnetic, and elastic wave equations with driven sources modeled by GMIG random fields. The inverse random source problem involves determining the strength function of the principal symbol of the covariance operator by utilizing the correlation of the far-field patterns of the random wave fields associated with the underlying wave equations. For the first time, we show that the optimal Lipschitz-type stability can be achieved for all considered the wave equations by employing the correlation-based data at a single frequency, a task that is unattainable for deterministic counterparts. The results demonstrate that random sources act as radiating sources, and the inverse random source problems exhibit stability through the utilization of correlation-based data.

The paper is organized as follows. Section \ref{sec:2} provides some preliminaries to GMIG random fields. Sections \ref{sec:PL} through \ref{sec:EL} focus on the stability estimates of inverse source problems for the stochastic polyharmonic, electromagnetic, and elastic wave equations, respectively. Finally, concluding remarks and perspectives for future research are provided in Section \ref{con}.

\section{Preliminaries}\label{sec:2}

This section presents a brief introduction to $\mathbb{R}^{d}$-valued GMIG random fields, which serve as the model for the random source in the subsequent sections. Without loss of generality, we assume that the random source discussed throughout the paper is a centered random field with zero mean. If this assumption does not hold, the random source can be centered by subtracting its non-zero mean.
 
\subsection{Random scalar field}

We begin with the scalar field case, i.e., $d=1$. Let $D\subset\mathbb{R}^d$ be an open domain with a Lipschitz boundary. The covariance operator $\Q_f$ for a random field $f\in \mathcal D'$ is defined as
\begin{align}\label{eq:Q1}
\langle\Q_f\phi,\psi\rangle:&=\mathbb E[\langle f,\phi\rangle\langle f,\psi\rangle]\notag\\
&=\int_{\mathbb R^d}\int_{\mathbb R^d}\mathbb E[f(y)f(x)]\phi(y)\psi(x)dydx\quad\forall\,\phi,\psi\in\mathcal D,
\end{align}
where $\mathcal D=\mathcal D(D;\mathbb R)$ is the space $C_0^\infty(D)$ equipped with some locally convex topology, known also as the space of test functions, its dual space $\mathcal D'$ is known as the space of distributions, and $\langle\cdot,\cdot\rangle$ denotes the dual product between $\mathcal D'$ and $\mathcal D$. 

\begin{definition}[cf. \cite{LPS08}]\label{def:GMIG}
A generalized Gaussian random field $f$ over $\mathbb{R}^d$ is termed microlocally isotropic of order $-m$ within a domain $D\subset\mathbb{R}^d$ when its covariance operator $\Q_f$ assumes the form of a classical pseudo-differential operator with an isotropic principal symbol $\sigma(x)|\xi|^{-m}$, where $\sigma\in C_0^\infty(D)$ represents the strength, being compactly supported in $D$, and $\sigma\geq 0$. 
\end{definition}

\begin{remark}
The regularity assumption on the strength $\sigma$ in Definition \ref{def:GMIG} can be relaxed to a weaker condition, where  $\sqrt\sigma\in W_0^{s,p}(D)$ for certain $s>0$ and $p>1$, as specified in Assumption \ref{assum:1}.
\end{remark}

For such a GMIG random field $f$ with its covariance operator $\Q_f$ being a pseudo-differential operator, let $c(\cdot,\cdot)$ denote the symbol of $\Q_f$. Then, we have
\begin{align}\label{eq:Q2}
\Q_f\phi(x)=\frac1{(2\pi)^d}\int_{\mathbb R^d}e^{{\rm i}x\cdot\xi}c(x,\xi)\widehat\phi(\xi)d\xi,
\end{align}
where 
\[
\widehat\phi(\xi)=\mathcal F[\phi](\xi)=\int_{\mathbb R^d}\phi(y)e^{-{\rm i}y\cdot\xi}dy
\]
denotes the Fourier transform of $\phi$, and $c(\cdot,\cdot)$ belongs to $\mathcal S^{-m}$, which is defined as
\begin{align*}
\mathcal S^{-m}=\mathcal S^{-m}(\mathbb R^d\times\mathbb R^d):=\left\{c\in C^\infty(\mathbb R^d\times\mathbb R^d)
:|\partial_x^\alpha\partial_\xi^\beta c(x,\xi)|\lesssim(1+|\xi|^2)^{\frac{-m-|\beta|}2}\quad\forall\,\alpha,\beta\in\mathbb Z_+^d\right\}.
\end{align*}

Henceforth, the notation $a\lesssim b$ indicates that $a\le Cb$ for some constant $C>0$, which may vary from line to line but is always independent of the wavenumber $k$ and the strength $\sigma$. 

Combining \eqref{eq:Q1} with \eqref{eq:Q2}, we obtain 
\begin{align*}
\Q_f\phi(x)=\int_{\mathbb R^d}\mathbb E[f(y)f(x)]\phi(y)dy=\int_{\mathbb R^d}\left[\frac1{(2\pi)^d}\int_{\mathbb R^d}e^{{\rm i}(x-y)\cdot\xi}c(x,\xi)d\xi\right]\phi(y)dy,
\end{align*}
and thus we derive the kernel $\mathbb E[f(y)f(x)]$ in the distribution sense as
\begin{align}\label{eq:kernel}
K_f(x,y):=\mathbb E[f(x)f(y)]=\frac1{(2\pi)^d}\int_{\mathbb R^d}e^{{\rm i}(x-y)\cdot\xi}c(x,\xi)d\xi.
\end{align}

\begin{assumption}\label{assum:1}
Consider $f$ as a real-valued centered GMIG random field with order $-m$, which is supported in the unit ball $B_1=\{x\in\mathbb R^d: |x|<1\}$, with $m\in(d+2-4n,d]$ and symbol given by 
\[
c(x,\xi)=\sigma(x)|\xi|^{-m}+r(x,\xi),
\]
where $\sigma\ge0$, $\sqrt\sigma\in W_0^{s,4}(B_1)$ with $s>\frac d4+2n-1$ being a positive integer, $r(\cdot,\xi)\in C_0^\infty(B_1)$, $r\in\mathcal S^{-m-1}$, and $n\in\mathbb N_+$.
\end{assumption}

\begin{remark}\label{rk:f}
Based on Assumption \ref{assum:1}, several observations can be made: 

\begin{itemize}

\item[(i)] The regularity assumption on $\sqrt{\sigma}$ implies that $\sigma\in H_0^s(B_1)$ due to the Leibniz formula
\[
\|\sigma\|_{H_0^s(B_1)}\lesssim\|\sqrt\sigma\|_{W_0^{s,4}(B_1)}^2.
\]
Furthermore, considering the embedding $W_0^{s,4}(B_1)\subset C_0(B_1)$ (cf. \cite[Theorem 7.57]{A75}) for any $s$ that satisfies the condition in Assumption \ref{assum:1}, it follows that $\sqrt{\sigma}\in C_0(B_1)$, and thus $\sigma\in C_0(B_1)$.

\item[(ii)] Since $r\in\mathcal S^{-m-1}$, there exists a constant $C_1>0$ depending only on the residual $r$ and the Lebesgue measure $\lambda(B_1)$ of $B_1$ such that
\begin{align}\label{eq:r}
\int_{B_1}|r(x,\xi)|dx\le C_1 |\xi|^{-m-1},\quad|\xi|>1.
\end{align} 

\end{itemize}
\end{remark}

The regularity of the random source $f$, satisfying Assumption \ref{assum:1}, is uniquely defined by its covariance operator, and thus determined by the order $-m$. As demonstrated in \cite{LW21}, the source $f\in C^{0,\alpha}(B_1)$ with $\alpha\in(0,\frac{m-d}2)$ is relatively regular if $m>d$. However, our primary emphasis lies in the scenario of rough sources where $m\leq d$, indicating that $f$ should be interpreted as a distribution.

\begin{lemma}\label{lem:reg_f}
Let $f$ be a GMIG random field satisfying Assumption \ref{assum:1}. Then $f\in H^{-\frac{d-m}2-\epsilon}(B_1)$ almost surely for any sufficiently small $\epsilon>0$.
\end{lemma}

\begin{proof}
In \cite[Lemma 2.6]{LW21},  it is shown that if $\tilde f$ is a GMIG random field having a smooth strength $\tilde\sigma\in C_0^{\infty}(B_1)$, then $\tilde f\in W^{-\frac{d-m}2-\frac\epsilon2,\tilde p}(B_1)$ for any sufficiently small $\epsilon>0$ and $\tilde p>1$.

For a random field $f$ satisfying Assumption \ref{assum:1}, its strength $\sigma$ satisfies
\[
\sqrt\sigma\in W_0^{s,4}(B_1)\subset W_0^{\frac{d-m}2+\frac\epsilon2,q'}(B_1)
\]
for any $s>\frac d4+2n-1$ and $q'\in(2,\infty)$. It is noted that $f$ exhibits the same regularity as $\sqrt{\sigma}\tilde f$, which satisfies
\[
\left\|\sqrt{\sigma}\tilde f\right\|_{W^{-\frac{d-m}2-
\frac\epsilon2,p'}(B_1)}\lesssim\left\|\sqrt{\sigma}\right\|_{W^{\frac{d-m}2+\frac\epsilon2,q'}(B_1)}\left\|\tilde f\right\|_{W^{-\frac{d-m}2-\frac\epsilon2,\tilde p}(B_1)}
\]
according to \cite[Lemma 2]{LPS08}, with $q'=\frac{2\tilde p}{\tilde p-1}\in(2,\infty)$ and $p'=\frac{2\tilde p}{\tilde p+1}\in(1,2)$ such that $\frac1{p'}+\frac1{q'}=1$. 

Choosing any $p'\in\left(\frac1{\frac12+\frac\epsilon{2d}},2\right)$, we derive from \cite[Theorem 7.63]{A75} that
\[
H_0^{\frac{d-m}2+\epsilon}(B_1)\subset W_0^{\frac{d-m}2+\frac\epsilon2,q'}(B_1),
\]
and thus
\begin{align*}
f\in W^{-\frac{d-m}2-\frac\epsilon2,p'}(B_1)\subset H^{-\frac{d-m}2-\epsilon}(B_1),
\end{align*}
which completes the proof.
\end{proof}

\subsection{Random vector field}\label{sec:2.2}

We introduce the $\mathbb{R}^d$-valued centered GMIG random vector field $\m{f} = (f_1, f_2, \ldots, f_d)^\top$ of order $-m$, where each component $f_j$, $j = 1, 2, \ldots, d$, constitutes a real-valued GMIG random field of identical order $-m$, defined in the domain $B_1$.  The random vector field $\m{f}$ is uniquely determined by its covariance operator $\mathcal Q_{\m{f}}$, defined as
\[
\langle\mathcal Q_{\m f}\m \phi,\m\psi\rangle=\int_{\mathbb R^d}\int_{\mathbb R^d}\m\phi^\top(y)\E\left[\m{f}(y) \m{f}(x)^\top\right]\m\psi(x)dydx\quad\forall\,\m\phi,\m\psi\in\mathcal D(D;\mathbb R^{d}),
\]
where
\[
\mathcal Q_{\m f}\m \phi(x)=\int_{\mathbb R^d}\mathbb E\left[\m f(x)\m f(y)^\top\right]\m\phi(y)dy=:\int_{\mathbb R^d}K_{\m f}(x,y)\m\phi(y)dy
\]
with the kernel $K_{\m{f}}(x,y) = \E\left[\m{f}(x) \m{f}(y)^\top\right]$ interpreted in the distribution sense.

\begin{assumption}\label{assump2}
Suppose $\boldsymbol{f}$ represents an $\mathbb{R}^{d}$-valued centered GMIG random field with order $-m$ over $B_1\subset{\mathbb{R}^d}$, where $m\in(d-2,d]$. The covariance operator $K_{\boldsymbol f}$ takes the form of a classical pseudo-differential operator, featuring the principal symbol $\boldsymbol{\sigma}(x)|\xi|^{-m}$, defined by  
\begin{eqnarray*}
K_{\m{f}} (x,y) = \frac{1}{(2\pi)^d} \int_{\mathbb{R}^d} e^{\ii(x-y)\cdot \xi}\boldsymbol{c}(x,\xi)d\xi,
\end{eqnarray*}
where the symbol $\boldsymbol{c}$ takes the form
\begin{eqnarray*}
\boldsymbol{c}(x,\xi) = \boldsymbol{\sigma}(x)|\xi|^{-m} + \boldsymbol{r}(x,\xi)
\end{eqnarray*}
with the strength matrix $\boldsymbol{\sigma}=[\sigma_{ij}]_{d\times d}$ being non-negative definite, $\boldsymbol{\sigma}^{\frac12}\in W_0^{s,4}(B_1;\mathbb R^{d\times d})$ for some positive integer $s>\frac d4+1$, $\boldsymbol{r}=[r_{ij}]_{d\times d}\in\mathcal S^{-m-1}(\mathbb{R}^d\times \mathbb{R}^d;\mathbb{R}^{d\times d})$, and $\boldsymbol{r}(\cdot,\xi)\in C_0^\infty (B_1;\mathbb{R}^{d\times d})$.
\end{assumption}

Similar to the scalar case discussed in Remark \ref{rk:f}, for $\m f$ satisfying Assumption \ref{assump2}, we can also deduce that $\boldsymbol\sigma\in H_0^s(B_1;\mathbb R^{d\times d})\cap C_0(B_1;\mathbb R^{d\times d})$, and that there exists a positive constant $C_2$ depending only on the residual $\boldsymbol r$ and $\lambda(B_1)$ such that
\begin{align}\label{eq:B}
\lt\Vert \int_{B_1} \boldsymbol{r}(x, \xi)e^{-\ii \xi \cdot x} dx\rt\Vert_F\le C_2|\xi|^{-m-1},\quad|\xi|>1,
\end{align}
where $\|\cdot\|_F$ denotes the Frobenius norm of a matrix.

\section{Polyharmonic waves}\label{sec:PL}

In the area of engineering and materials science, the polyharmonic wave equation plays an important role for analyzing the vibrational modes and dynamic response of mechanical systems and composite materials, facilitating the design and optimization of structures for various applications, including aerospace and automotive engineering \cite{GGS10}. In particular, the Helmholtz equation and the biharmonic wave equation represent two specific instances of the polyharmonic wave equation. This section examines both the direct and inverse problems for the polyharmonic wave equation. Related inverse boundary value problems concerning the polyharmonic operator are discussed in \cite{KLU-TAMS14}. 

Consider the stochastic polyharmonic wave equation
\begin{align}\label{eq:poly}
(-\Delta)^nu-k^{2n}u=f\quad\text{in} ~ \mathbb R^d,
\end{align}
where $d=2,3$, $n\ge1$ is an integer, $k>0$ is the wavenumber,  and the source $f$ is a GMIG random field satisfying Assumption \ref{assum:1}. The Sommerfeld radiation condition is imposed on the wave field $u$:
\begin{equation}\label{eq:radia}
\lim_{r=|x|\to\infty}r^{\frac{d-1}2}\left(\partial_ru-{\rm i}ku\right)=0.
\end{equation}

When $n=1$, the polyharmonic wave equation \eqref{eq:poly} reduces to the Helmholtz equation 
\begin{align}\label{eq:acoustic}
\Delta u+k^2u=-f\quad\text{in} ~ \mathbb R^d,
\end{align}
which describes acoustic wave phenomena, as studied in \cite{LW21}. When $n=2$, the polyharmonic wave equation transforms into the biharmonic wave equation, as investigated in \cite{LW22}.

It is noteworthy that for any $n\geq1$, only a single radiation condition \eqref{eq:radia} is necessary, similar to the case of the biharmonic equation with $n=2$, as elaborated in \cite{BH20}. This arises from the capability to decompose the wave field $u$ into two components: a propagating wave component, which is the solution to the Helmholtz equation \eqref{eq:acoustic}, and an exponentially decaying wave component, which satisfies the modified Helmholtz equation, where the wavenumber $k$ is replaced by ${\rm i}k$ in \eqref{eq:acoustic}. The latter component does not necessitate the radiation condition.

\subsection{The direct problem}

The well-posedness has been established for the direct problems associated with both the Helmholtz equation \cite{LW21} and the biharmonic wave equation \cite{LW22}. Here, we extend this analysis to investigate the well-posedness of the direct problem concerning the polyharmonic wave equation \eqref{eq:poly}--\eqref{eq:radia} for any $n\geq1$.

It can be verified from the operator splitting that
\begin{equation}\label{os}
\left((-\Delta)^n-k^{2n}\right)^{-1}=\frac1{nk^{2n}}\sum_{j=0}^{n-1}\kappa_j^2(-\Delta-\kappa_j^2)^{-1},
\end{equation}
where $\kappa_j=ke^{{\rm i}\frac{j\pi}n}, j=0,1,\cdots,n-1$. Let $G$ denote the Green's function  corresponding to the polyharmonic wave operator $(-\Delta)^n-k^{2n}$, which satisfies
\[
((-\Delta)^n-k^{2n})G(x,y,k)=-\delta(x-y). 
\] 
Utilizing the operator splitting \eqref{os}, the Green's function $G$ takes the form
\begin{equation}\label{eq:Green1}
G(x,y,k)=\frac1{nk^{2n}}\sum_{j=0}^{n-1}\kappa_j^2\Phi_d(x,y,\kappa_j),
\end{equation}
where $\Phi_d$ is the Green's function of the Helmholtz equation \eqref{eq:acoustic}, explicitly defined as 
\begin{equation*}
\Phi_d(x,y,\kappa_j)=\left\{
\begin{aligned}
&\frac{\rm i}{4}H_0^{(1)}(\kappa_j|x-y|),\quad &d=2,\\
&\frac1{4\pi}\frac{e^{{\rm i}\kappa_j|x-y|}}{|x-y|},\quad &d=3.
\end{aligned}
\right.
\end{equation*}
Here, $H_0^{(1)}$ represents the Hankel function of the first kind with order 0. 

Define an integral operator 
\begin{align*}
(\mathcal H_k\phi)(x):=\int_{\mathbb R^d}G(x,y,k)\phi(y)dy,
\end{align*}
which admits the following decaying properties. The proof is deferred to the appendix to maintain focus on the main results without distraction.
 
\begin{lemma}\label{lm:H}
Suppose $D$ and $B$ are two bounded Lipschitz domains of $\mathbb{R}^d$. The operator $\mathcal H_k:H^{-s_1}(D)\to H^{s_2}(B)$ is bounded, satisfying
\begin{align}\label{eq:H1}
\|\mathcal H_k\|_{\mathcal L(H^{-s_1}(D),H^{s_2}(B))}\lesssim k^{s-2n+1}
\end{align}
for $s:=s_1+s_2\in[0,2n-1)$ with $s_1,s_2\ge0$. Furthermore, this boundedness extends to $\mathcal H_k:H^{-s_1}(D)\to L^\infty(B)$ with
\begin{align}\label{eq:H2}
\|\mathcal H_k\|_{\mathcal L(H^{-s}(D),L^\infty(B))}\lesssim k^{s-2n+1+\frac{d}{2}+\epsilon}
\end{align}
for any $s\in[0,2n-1)$ and $\epsilon>0$.
\end{lemma}

The following result demonstrates the well-posedness of the direct problem \eqref{eq:poly}--\eqref{eq:radia} in the distributional sense.

\begin{theorem}\label{wp-ph}
The problem \eqref{eq:poly}--\eqref{eq:radia} admits a unique distributional solution
\begin{equation}\label{eq:solu1}
u:=-\mathcal H_kf\in H^{2n-1-\frac{d-m}{2}-\epsilon}_{loc}(\mathbb{R}^d) 
\end{equation}
for any $\epsilon>0$, which satisfies
\[
\langle(-\Delta)^nu,v\rangle-k^{2n}\langle u,v\rangle=\langle f,v\rangle\quad\forall\,v\in \mathcal D.
\]
\end{theorem}

\begin{proof}
It follows from Lemma \ref{lem:reg_f} that $f\in H^{-\frac{d-m}2-\epsilon}(B_1)$. By Lemma \ref{lm:H}, it holds that
\[
u=-\mathcal H_kf\in H^{2n-1+\frac{m-d}{2}-\epsilon}_{loc}(\mathbb{R}^d)
\] 
for any $\epsilon>0$, where $2n-1+\frac{m-d}{2}>0$ for $m\in(d+2-4n,d]$. For any $v\in C_0^\infty(\mathbb{R}^d)$, we have from the integration by parts that 
\begin{align*}
\langle (-\Delta)^nu-k^{2n}u,v\rangle&=\langle u,(-\Delta)^nv\rangle-k^{2n}\langle u,v\rangle\\
&=-\int_{\mathbb R^d}\langle G(\cdot,y,k),((-\Delta)^n-k^{2n})v\rangle f(y)dy\\
&=\langle f,v\rangle,
\end{align*}
which indicates that $u=-\mathcal H_kf$ is a distributional solution of \eqref{eq:poly}--\eqref{eq:radia}.
\end{proof}

As demanded in Assumption \ref{assum:1}, the well-posedness in Theorem \ref{wp-ph} is ensured under the condition $m\in(d+2-4n,d]$, which aligns with the requirement stated in \cite[Theorem 3.2]{LW21} for $n=1$ and the one specified in \cite[Assumption 2.3]{LW22} for $n=2$.

Utilizing \eqref{eq:Green1}, the Green's function $G$ has the asymptotic behavior
\begin{equation}\label{Ga}
G(x,y,k)= \frac{e^{\ii k |x|}}{|x|^{\frac{d-1}{2}}} \left( \frac{\Cd}{n} k^{\frac{d+1-4n}{2}}e^{-\ii k \hat{x}\cdot y}+ O(|x|^{-1})\right),\quad|x|\to \infty,
\end{equation}
where $\hat{x} := \frac{x}{|x|} \in \mathbb{S}^{d-1}$ and
\begin{equation}\label{eq:beta}
\Cd=\left\{
\begin{aligned}
&\frac{e^{\ii \frac{\pi}{4}}}{\sqrt{8\pi}},\quad& d=2,\\
&\frac{1}{4\pi},\quad& d=3.
\end{aligned}
\right.
\end{equation}
Combining \eqref{eq:solu1} and \eqref{Ga} yields 
\[
u(x,k) = \frac{e^{\ii k |x|}}{|x|^{\frac{d-1}{2}}} \left(u^\infty(\hat{x}) +O(|x|^{-1})\right), 
\]
where the far-field pattern is 
\begin{equation}\label{eq:farfield}
u^\infty(\hat{x},k) = -\frac{\Cd}{n} k^{\frac{d+1-4n}{2}}\int_{\mathbb{R}^d}e^{\ii k\hat{x}\cdot y}f(y)dy,\quad \hat x\in \mathbb S^{d-1}.
\end{equation}

\begin{theorem}\label{tm:u_infty}
The far-field pattern $u^\infty$ defined in \eqref{eq:farfield} is bounded almost surely.
\end{theorem}

\begin{proof}
For each $\hat{x}\in\mathbb S^{d-1}$, we have from \eqref{eq:farfield} that
\begin{align*}
 |u^\infty(\hat{x},k)|&=\left|\frac{\Cd}n k^{\frac{d+1-4n}{2}}\int_{\mathbb{R}^d} e^{-\ii k\hat{x}\cdot y}f(y)dy\right|\\
 &\lesssim \Cd k^{\frac{d+1-4n}{2}} \Vert p(k,\hat{x},\cdot)\Vert_{H^{\frac{d-m}{2}+\epsilon}(B_1)} \Vert f \Vert_{H^{-\frac{d-m}{2}-\epsilon}(B_1)},
\end{align*}
where  $p(k,\hat{x},y) := e^{-\ii k\hat{x}\cdot y}$. By the regularity of the random source presented in Lemma \ref{lem:reg_f}, we conclude that  $\|u^\infty(\cdot,k)\|_{L^\infty(\mathbb S^{d-1})}$ is bounded almost surely.
\end{proof}

The far-field pattern $u^\infty$ of the wave field $u$ serves as data to determine the strength $\sigma$ associated with the covariance operator of the random source $f$.

\subsection{The inverse problem}

Given $\hat{x},\hat{y}\in\mathbb S^{d-1}$, define the correlation function of the far-field pattern as
\begin{align*}
F_{\rm PL}(\hat{x},\hat{y},k):=\E[u^\infty(\hat{x},k) u^\infty(\hat{y},k)].
\end{align*}
By Theorem \ref{tm:u_infty}, we may denote 
\begin{equation*}
 M_{\rm PL}(k) :=\Vert F_{\rm PL}(\cdot,\cdot,k)\Vert_{L^\infty(\mathbb S^{d-1}\times \mathbb S^{d-1})}.
 \end{equation*}
 
To derive the stability estimate, we begin with the estimates of $M_{\rm PL}(k)$ and the low-frequency component of the Fourier coefficient of the strength function $\widehat \sigma$. 

\begin{lemma}\label{lm:M_PL}
For any wavenumber $k>1$, the quantity $M_{\rm PL}(k)$ is well-defined and satisfies the estimates
\begin{align*}
\frac{\Cd^2}{n^2}\left(\|\sigma\|_{L^1(B_1)}-C_1 k^{-1}\right) k^{d+1-4n-m}\le M_{\rm PL}(k)\le \frac{\Cd^2}{n^2}\left(\|\sigma\|_{L^1(B)}+C_1 k^{-1}\right) k^{d+1-4n-m},
\end{align*}
where $C_1$ is given in \eqref{eq:r}.
\end{lemma}

\begin{proof}
For any $\hat x,\hat y\in\mathbb S^{d-1}$ and $k>1$, we obtain from \eqref{eq:kernel} that 
\begin{align}\label{eq:F}
F_{\rm PL}(\hat x,\hat y,k)
&=\frac{\Cd^2}{n^2} k^{d+1-4n} \int_{\mathbb{R}^d} \int_{\mathbb{R}^d} e^{-\ii k\hat{x}\cdot z_1} e^{-\ii k\hat{y}\cdot z_2}K_f(z_1,z_2) dz_1 dz_2\notag\\
&=\frac{\Cd^2}{n^2} k^{d+1-4n}\int_{\mathbb{R}^d} \int_{\mathbb{R}^d} e^{-\ii k\hat{x}\cdot z_1} e^{-\ii k\hat{y}\cdot z_2} \left[\frac{1}{(2\pi)^d} \int_{\mathbb{R}^d} e^{\ii (z_1-z_2)\cdot \xi}c(z_1, \xi)d\xi\right] dz_1 dz_2 \notag\\
&=\frac{\Cd^2}{n^2(2\pi)^d} k^{d+1-4n}\int_{\mathbb{R}^d} \int_{\mathbb{R}^d} \left[\int_{\mathbb{R}^d}  e^{-\ii (k\hat{y}+\xi)\cdot z_2} dz_2\right]c(z_1,\xi)e^{-\ii (k\hat{x}-\xi)\cdot z_1} d\xi dz_1\notag\\
&=\frac{\Cd^2}{n^2} k^{d+1-4n}\int_{\mathbb{R}^d} \int_{\mathbb{R}^d}\delta(\xi+k\hat{y})\left(\sigma(z_1)|\xi|^{-m}+r(z_1, \xi)\right)e^{-\ii (k\hat{x}-\xi)\cdot z_1} d\xi dz_1\notag\\
&=\frac{\Cd^2}{n^2} k^{d+1-4n}\left[\int_{\mathbb{R}^d} e^{-\ii k(\hat{x}+\hat{y})\cdot z_1} \sigma(z_1)k^{-m} dz_1+\int_{\mathbb{R}^d} r(z_1, -k \hat{y})e^{-\ii k(\hat{x}+\hat{y})\cdot z_1} dz_1\right],
\end{align}
where $r(\cdot,\xi)\in C_0^\infty(B_1)$ for any $\xi\in\mathbb R^d$, and $r\in\mathcal S^{-m-1}$ satisfying
\begin{align*}
\int_{B_1}|r(z_1,-k\hat y)|dz_1\le C_1 k^{-m-1}
\end{align*}
as stated in \eqref{eq:r}.

By Remark \ref{rk:f}, we have $\sigma\in C_0(B_1)$, thereby implying 
\begin{align*}
M_{\rm PL}(k)&=\sup_{\hat x,\hat y\in\mathbb S^{d-1}}\left|F_{\rm PL}(\hat x,\hat y,k)\right|\\
&=\frac{\Cd^2}{n^2} k^{d+1-4n}\sup_{\hat x,\hat y\in\mathbb S^{d-1}}\left|k^{-m}\int_{B_1} e^{-\ii k(\hat{x}+\hat{y})\cdot z_1} \sigma(z_1) dz_1+\int_{B_1} r(z_1, -k \hat{y})e^{-\ii k(\hat{x}+\hat{y})\cdot z_1} dz_1\right|\\
&\le\frac{\Cd^2}{n^2}k^{d+1-4n-m}\left(\|\sigma\|_{L^1(B_1)}+C_1 k^{-1}\right).
\end{align*}
On the other hand, choosing $\hat y=-\hat x$ in the above supreme norm yields 
\begin{align*}
M_{\rm PL}(k)&=\frac{\Cd^2}{n^2} k^{d+1-4n}\sup_{\hat x,\hat y\in\mathbb S^{d-1}}\left|k^{-m}\int_{B_1} e^{-\ii k(\hat{x}+\hat{y})\cdot z_1} \sigma(z_1) dz_1+\int_{B_1} r(z_1, -k \hat{y})e^{-\ii k(\hat{x}+\hat{y})\cdot z_1} dz_1\right|\\
&\ge\frac{\Cd^2}{n^2} k^{d+1-4n}\left|k^{-m}\int_{B_1}\sigma(z_1) dz_1+\int_{B_1} r(z_1, k \hat{x})dz_1\right|\\
&\ge\frac{\Cd^2}{n^2} k^{d+1-4n}\left[k^{-m}\int_{B_1}\sigma(z_1) dz_1-\int_{B_1}|r(z_1,k\hat x)|dz_1\right]\\
&\ge\frac{\Cd^2}{n^2} k^{d+1-4n-m}\left(\|\sigma\|_{L^1(B_1)}-C_1 k^{-1}\right),
\end{align*}
which completes the proof.
\end{proof}
\begin{lemma}\label{lm:Fourier_PL}
For any $k>1$ and $\gamma \in \mathbb{R}^d$ with $|\gamma|\leq 2k$, the Fourier coefficient $\widehat\sigma(\gamma)$ satisfies
\[
|\widehat{\sigma}(\gamma)|\le \frac{n^2}{\Cd^2} k^{m+4n-d-1}M_{\rm PL}(k)+C_1 k^{-1}. 
\]
\end{lemma}

\begin{proof}
For any fixed $\gamma \in \mathbb{R}^d$ with $|\gamma|\leq 2k$, there exists a unit vector $d_1$ such that $d_1\cdot \gamma=0$. It can be verified that unit vectors $\hat{x} = \frac{\gamma + \sqrt{4k^2-|\gamma|^2}d_1}{2k}\in\mathbb S^{d-1}$ and $\hat{y} = \frac{\gamma - \sqrt{4k^2-|\gamma|^2}d_1}{2k}\in\mathbb S^{d-1}$ satisfy $\gamma=k(\hat{x}+\hat{y} )$. Substituting $\hat x$ and $\hat y$ into \eqref{eq:F} leads to 
\begin{align*}
F_{\rm PL}(\hat x,\hat y,k)
&=\frac{\Cd^2}{n^2} k^{d+1-4n}\left[\int_{\mathbb{R}^d} e^{-\ii k(\hat{x}+\hat{y})\cdot z_1} \sigma(z_1)k^{-m} dz_1+\int_{\mathbb{R}^d} r(z_1, -k \hat{y})e^{-\ii k(\hat{x}+\hat{y})\cdot z_1} dz_1\right]\\
&=\frac{\Cd^2}{n^2} k^{d+1-4n}\left[k^{-m}\widehat{\sigma}( \gamma)+\int_{B_1} r(z_1, -k\hat{y})e^{-\ii \gamma\cdot z_1} dz_1\right].
\end{align*}
Noting $r(\cdot,\xi)\in C_0^\infty(B_1)$ for any $\xi\in\mathbb R^d$ and $r\in\mathcal S^{-m-1}$, we have from \eqref{eq:r} that 
\begin{equation*}
\left|\int_{B_1} r(z_1, -k\hat{y})e^{-\ii \gamma \cdot z_1} dz_1\right|\le\int_{B_1} |r(z_1, -k\hat{y})| dz_1\le C_1 k^{-m-1}.
\end{equation*}
By combining the above estimates, it follows for any $\gamma\in\mathbb{R}^d$ with $|\gamma|\leq 2k$ that
\begin{align*}
|\widehat{\sigma}(\gamma)|&\le \frac{n^2}{\Cd^2} k^{m+4n-d-1}\left|F_{\rm PL}(\hat x,\hat y,k)\right|+C_1 k^{-1}\\
&\le \frac{n^2}{\Cd^2} k^{m+4n-d-1}M_{\rm PL}(k)+C_1 k^{-1},
\end{align*}
which completes the proof.
\end{proof}

The following result demonstrates the Lipschitz stability of the inverse random source problem by utilizing the correlation of the far-field patterns associated with the wave field of the polyharmonic wave equation.

\begin{theorem}\label{stab_poly}
Let $f$ satisfy Assumption \ref{assum:1} with the additional condition $s>\max\{\frac d4+2n-1,d\}$. Then there exists a constant $k_0>1$ such that for any $k > k_0$, the following estimate holds:
\begin{eqnarray*}
\Vert \sigma\Vert_{L^\infty(B_1)} \lesssim k^{\frac ds+m+4n-d-1}\left(1+\|\sigma\|_{H^s(B_1)}\right)M_{\rm PL}(k).
\end{eqnarray*}
\end{theorem}

\begin{proof}
It is clear to note that
\begin{align*}
\Vert \sigma\Vert_{L^\infty(B_1)} &= \sup_{x\in B_1}\left\vert \frac1{(2\pi)^d}\int_{\mathbb{R}^d} e^{\ii\gamma\cdot x} \widehat{\sigma}(\gamma)d\gamma\right\vert\\
&\leq \frac1{(2\pi)^d}\int_{|\gamma|\leq \rho} |\widehat{\sigma}(\gamma)|d\gamma+ \frac1{(2\pi)^d}\int_{|\gamma|>\rho} |\widehat{\sigma}(\gamma)|d\gamma\\
&=: I_1(\rho) + I_2(\rho).
\end{align*}
The low-frequency component $I_1(\rho)$ can be estimated using Lemma \ref{lm:Fourier_PL}. Specifically, for any $\rho\le2k$, we have 
\begin{eqnarray}\label{eq:I_1}
I_1(\rho)=\frac1{(2\pi)^d}\int_{|\gamma|\le\rho}|\widehat\sigma(\gamma)|d\gamma
\le \frac{\lambda(B_\rho)}{(2\pi)^d}\left[\frac{n^2}{\Cd^2} k^{m+4n-d-1}M_{\rm PL}(k)+C_1 k^{-1}\right],
\end{eqnarray}
where the Lebesgue measure $\lambda(B_\rho)$ is given by $
\lambda(B_\rho)=\frac{\pi^{\frac d2}}{\Gamma(1+\frac d2)}\rho^d.$

For the high-frequency component $I_2(\rho)$, by the Paley--Wiener--Schwartz theorem \cite{hormander03}, we obtain for any $\sigma\in H_0^s(B_1)$ that
\begin{equation*}
|\widehat{\sigma}(\gamma)|\lesssim (1+|\gamma|)^{-s}\|\sigma\|_{H^s(B_1)}.
\end{equation*}
Consequently, it holds for $s>d$ that
\begin{equation}\label{eq:I_2}
I_2(\rho)=\frac1{(2\pi)^d}\int_{|\gamma|>\rho} |\widehat{\sigma}(\gamma)|d\gamma\lesssim\|\sigma\|_{H^s(B_1)}\int_{|\gamma|>\rho}(1+|\gamma|)^{-s}d\gamma \lesssim \|\sigma\|_{H^s(B_1)}\rho^{d-s}.
\end{equation}
Combining \eqref{eq:I_1} and \eqref{eq:I_2}, we derive for any $k>1$ and $\rho\le 2k$ that 
\begin{eqnarray}\label{eq:sigma_PL}
\Vert \sigma\Vert_{L^\infty(B_1)} \lesssim \rho^d\left( k^{m+4n-d-1}M_{\rm PL}(k) + C_1 k^{-1}+ \|\sigma\|_{H^s(B_1)}\rho^{-s}\right). 
\end{eqnarray}

Since $\sigma\in C_0(B_1)$, $\sigma\ge0$, and $\sigma\not\equiv0$, it follows that $\|\sigma\|_{L^1(B_1)}>0$, and one can pick $k_0>\max\left\{\frac{2C_1}{\|\sigma\|_{L^1(B_1)}},1\right\}\geq1$. For any $k>k_0$, we have 
\[
k^{-1}<k_0^{-1}<\frac{\|\sigma\|_{L^1(B_1)}}{2C_1},
\]
which, together with the estimate stated in Lemma \ref{lm:M_PL}, yields 
\[
k^{d+1-m-4n}\le\frac{2n^2}{\Cd^2\|\sigma\|_{L^1(B_1)}} M_{\rm PL}(k).
\] 
Choosing $\rho=k^{\frac1s}<2k$ in \eqref{eq:sigma_PL}, we obtain 
\begin{align*}
\Vert \sigma\Vert_{L^\infty(B_1)}
&\lesssim k^{\frac ds+m+4n-d-1}\left(M_{\rm PL}(k)+ C_1 k^{-1}k^{d+1-m-4n}+\|\sigma\|_{H^s(B_1)}k^{-1}k^{d+1-m-4n}\right)\\
&\lesssim k^{\frac ds+m+4n-d-1}\left(M_{\rm PL}(k)+\frac{n^2}{\Cd^2}M_{\rm PL}(k)+\frac{n^2\|\sigma\|_{H^s(B_1)}}{\Cd^2C_1}M_{\rm PL}(k)\right)\\
&\lesssim k^{\frac ds+m+4n-d-1}\left(1+\|\sigma\|_{H^s(B_1)}\right)M_{\rm PL}(k),
\end{align*}
which completes the proof.
\end{proof}

\begin{remark}
A stability estimate of the strength function $\sigma$ in $L^1$ norm can be directly derived from Lemma \ref{lm:M_PL}. Given any $k>k_0$ where $k_0>\max\left\{\frac{2C_1}{\|\sigma\|_{L^1(B_1)}},1\right\}\geq1$, Lemma \ref{lm:M_PL} implies that 
\[
\|\sigma\|_{L^1(B_1)} \le  k^{4n+m-d-1}  \frac{2n^2}{\Cd^2}M_{\rm PL}(k).
\]
\end{remark}

\section{Electromagnetic waves}\label{sec:EM}

This section aims to examine the stability of the inverse random source problem of Maxwell's equations. We demonstrate that Lipschitz-type stability can be extended from the scalar case of polyharmonic waves to the vector case of electromagnetic waves, provided that the weak divergence-free condition is satisfied.

Consider the stochastic Maxwell's equations
\begin{equation}\label{Me}
\nabla\times \m{E} -\ii k \m{H}=0,\quad \nabla\times \m{H}+\ii k\m{E} = \m{f}\quad \text{ in }\mathbb{R}^3,
\end{equation}
where $k>0$ denotes the wavenumber, $\m{E}$ is the electric field, and $\m{H}$ represents the magnetic field. The electromagnatic fields $\m{E}$ and $\m{H}$ satisfy the following weak Silver--M\"{u}ller radiation conditions proposed in \cite{LW21Max}:
\begin{equation*}
\lim_{r\to\infty} \int_{|x|=r}(\m{H}\times \hat{x}-\m{E})\cdot {\m\phi} ds=0\quad \forall\,  \m\phi\in\m{\mathcal{D}}:=\mathcal D^3.
\end{equation*}

Assume that the random source $\m{f}$, representing the electric current density, satisfies Assumption \ref{assump2} with $d=3$ such that $m\in (-1,3]$ (cf. \cite{LW21Max}), and belongs to the space
\begin{equation*}
\mathbb{X}:=\left\{ \m{U}\in \m{\mathcal{D}}': \int_{\mathbb{R}^3} \m{U}\cdot (\nabla(\nabla\cdot \m\phi))dx=0\quad\forall\, \m\phi\in\m{\mathcal{D}} \right\},
\end{equation*}
which denotes the space of all distributions that are divergence-free in the sense of distributions. 

By eliminating the magnetic field $\m{H}$ in \eqref{Me}, it is shown in \cite{LW21Max} that the above assumptions on $\m{f}$ ensure the reduction of Maxwell's equations to the Helmholtz equation for the electric field $\m{E}$:
\begin{equation}\label{vHe}
\Delta\m E+k^2\m E=-{\rm i}k\m f.
\end{equation}
Furthermore, the electric field has an integral representation in terms of the source, where the integral kernel corresponds precisely to the fundamental solution of the Helmholtz equation. 

Comparing \eqref{eq:poly} with \eqref{vHe} and setting $n=1$ in \eqref{eq:farfield}, we derive the far-field pattern of the electric field:
\begin{equation*}
\m{E}^\infty(\hat{x},k) = \ii k \beta_3 \int_{\mathbb{R}^3}e^{-\ii k \hat{x}\cdot y}\m{f}(y)dy
\end{equation*}
where $\beta_3$ is given in \eqref{eq:beta}. The boundedness of the far-field pattern $\m E^\infty$ can be obtained similarly as in Theorem \ref{tm:u_infty}.

Using the measured far-field patterns as data, we define  
\begin{eqnarray*}
F_{\rm EM}(\hat{x},\hat{y},k):=\left\Vert\E\left[\m{E}^\infty(\hat{x},k) \m{E}^\infty(\hat{y},k)^\top\right]\right\Vert_F
\end{eqnarray*}
and
\begin{equation*}\label{eq:defM_elec}
M_{\rm EM}(k) := \Vert F_{\rm EM}(\cdot,\cdot,k)\Vert_{L^\infty(\mathbb S^2\times\mathbb S^2)}. 
\end{equation*}

In the following, we use notations
\[
\|\m A\|_{L^{\infty}(B;\mathbb F^{3\times3})}:=\sup_{x\in B}\|\m A(x)\|_F=\sup_{x\in B}\left(\sum_{i,j=1}^3|a_{ij}(x)|^2\right)^{\frac12}=\left(\sum_{i,j=1}^3\|a_{ij}\|_{L^\infty(B)}^2\right)^{\frac12}
\]
and 
\[
\|\m A\|_{H^s(B;\mathbb F^{3\times3})}:=\left(\sum_{i,j=1}^3\|a_{ij}\|_{H^s(B)}^2\right)^{\frac12}
\]
for any $\mathbb{F}^{3\times3}$-valued function $\m A=[a_{ij}]_{3\times3}$ defined in the domain $B$, where $\mathbb{F}^{3\times3}$ denotes the $3\times3$ matrix space over the field $\mathbb{F}=\mathbb{R}$ or $\mathbb{C}$ equipped with the Frobenius norm $\|\cdot\|_F$.

Similarly, we derive the following estimate for the low-frequency component of the Fourier coefficient matrix $\widehat{\m\sigma}(\gamma)$, which subsequently yields the stability estimate for the reconstruction of the strength matrix $\m\sigma$.

\begin{lemma}\label{lm:Fourier_EM}
Let $\m f\in\mathbb X$ satisfy Assumption \ref{assump2} with $d=3$. For any $k>1$ and $\gamma \in \mathbb{R}^3$ with $|\gamma|\leq 2k$, the Fourier coefficient matrix $\widehat{\m\sigma}$ satisfies
\[
\Vert\widehat{\m{\sigma}}(\gamma)\Vert_F\le \frac1{\beta_3^2}k^{m-2}M_{\rm EM}(k)+  C_2 k^{-1},
\]
where $C_2$ is given in \eqref{eq:B}.
\end{lemma}

\begin{proof}
For any fixed $\gamma \in \mathbb{R}^3$ with $|\gamma|\leq 2k$, there exists a unit vector $d_1$ such that $d_1\cdot \gamma=0$. It is clear to note that $\hat{x} = \frac{\gamma + \sqrt{4k^2-|\gamma|^2}d_1}{2k}, \hat{y} = \frac{\gamma - \sqrt{4k^2-|\gamma|^2}d_1}{2k}$ are unit vectors in $\mathbb S^2$ and satisfy $k(\hat{x}+\hat{y} )=\gamma$. A straightforward calculation yields 
  \begin{align}\label{eq:E}
   -\frac1{k^2\beta_3^2} \E\left[\m{E}^\infty(\hat{x},k) \m{E}^\infty(\hat{y},k)^\top\right]
   &= \int_{\mathbb{R}^3} \int_{\mathbb{R}^3} e^{-\ii k\hat{x}\cdot z_1} e^{-\ii k\hat{y}\cdot z_2}K_{\m{f}}(z_1,z_2) dz_1 dz_2\notag\\
    &=\int_{\mathbb{R}^3} \int_{\mathbb{R}^3} e^{-\ii k\hat{x}\cdot z_1} e^{-\ii k\hat{y}\cdot z_2}\left[\frac1{(2\pi)^3}\int_{\mathbb R^3} e^{\ii (z_1-z_2)\cdot \xi}{\boldsymbol c}(z_1, \xi)d\xi\right] dz_1 dz_2 \notag\\
    &=\frac1{(2\pi)^3}\int_{\mathbb{R}^3} \int_{\mathbb{R}^3} \left( \int_{\mathbb{R}^3}  e^{\ii (-k\hat{y}-\xi)\cdot z_2} dz_2 \right){\boldsymbol c}(z_1, \xi)e^{-\ii (k\hat{x}-\xi)\cdot z_1} d\xi dz_1\notag\\
     &=\int_{\mathbb{R}^3} \int_{\mathbb{R}^3}\delta(\xi+k\hat{y})\left(\m{\sigma}(z_1)|\xi|^{-m}+{\boldsymbol r}(z_1, \xi)\right)e^{-\ii (k\hat{x}-\xi)\cdot z_1} d\xi dz_1\notag\\
       &=\left[\int_{\mathbb{R}^3} e^{-\ii k(\hat{x}+\hat{y})\cdot z_1} \m{\sigma}(z_1)k^{-m} dz_1+\int_{B_1} {\boldsymbol r}(z_1, -k\hat{y})e^{-\ii k(\hat{x}+\hat{y})\cdot z_1} dz_1\right]\notag\\
       &= k^{-m}\widehat{\m{\sigma}}( \gamma)+ \int_{B_1}{\boldsymbol r}(z_1, -k\hat{y})e^{-\ii \gamma\cdot z_1} dz_1,
    \end{align}
where
\begin{equation*}
\lt\Vert \int_{B_1} {\boldsymbol r}(z_1, -k\hat{y})e^{-\ii \gamma \cdot z_1} dz_1\rt\Vert_F\le C_2 k^{-m-1}
\end{equation*}
according to \eqref{eq:B}. Thus for any $\gamma\in\mathbb{R}^3$ with $|\gamma|\leq 2k$, we obtain 
\begin{align*}
\left\|\widehat{\m{\sigma}}(\gamma)\right\|_F
&\le \frac1{\beta_3^2}k^{m-2} \lt\Vert\E\left[\m{E}^\infty(\hat{x},k)\m{E}^\infty(\hat{y},k)^\top\right]\rt\Vert_F+k^m\lt\Vert \int_{B_1} {\boldsymbol r}(z_1, -k\hat{y})e^{-\ii \gamma \cdot z_1} dz_1\rt\Vert_F\\
&\le \frac1{\beta_3^2}k^{m-2}F_{\rm EM}(\hat x,\hat y,k)+  C_2 k^{-1}\\
&\le\frac1{\beta_3^2}k^{m-2}M_{\rm EM}(k)+  C_2 k^{-1},
\end{align*}
which completes the proof.
\end{proof}

Utilizing Lemma \ref{lm:Fourier_EM}, we can derive the stability result through a similar argument as presented in the proof of Theorem \ref{stab_poly}. 

\begin{theorem}
Let $\m f\in\mathbb X$ satisfy Assumption \ref{assump2} with $d=3$ and additionally $s>3$. There exists a constant $k_0>1$ such that for any $k > k_0$, the following estimate holds:
\begin{equation*}
\Vert \m{\sigma}\Vert_{L^\infty(B_1;\mathbb R^{3\times3})} \lesssim k^{\frac3s+m-2}\left(1+\|\m\sigma\|_{H^s(B_1;\mathbb R^{3\times3})}\right)M_{\rm EM}(k).
\end{equation*}
\end{theorem}

\begin{proof}
Note that
\begin{align*}
\Vert \m{\sigma}\Vert_{L^\infty(B_1;\mathbb R^{3\times3})} 
&=\sup_{x\in B_1}\left\|\frac1{(2\pi)^3}\int_{\mathbb R^3}\widehat{\m\sigma}(\gamma)e^{-{\rm i}x\cdot\gamma}d\gamma\right\|_F\\
&\le \frac1{(2\pi)^3}\int_{|\gamma|\leq \rho} \|\widehat{\m{\sigma}}(\gamma)\|_Fd\gamma+ \frac1{(2\pi)^3}\int_{|\gamma|>\rho} \|\widehat{\m{\sigma}}(\gamma)\|_Fd\gamma\\ 
&=: I_1(\rho)+I_2(\rho). 
\end{align*}
We have from Lemma \ref{lm:Fourier_EM} that 
\[
I_1(\rho)\le\frac{\lambda(B_\rho)}{(2\pi)^3}\left[\frac1{\beta_3^2}k^{m-2}M_{\rm EM}(k)+  C_2 k^{-1}\right]. 
\]
For $s>3$, it follows from the Paley--Wiener--Schwartz theorem used in the proof of Theorem \ref{stab_poly} that 
\begin{align*}
I_2(\rho)
&=\frac1{(2\pi)^3}\int_{|\gamma|>\rho}\left(\sum_{i,j=1}^3 \left|\widehat\sigma_{ij}(\gamma)\right|^2\right)^{\frac12}d\gamma\\
&\lesssim\int_{|\gamma|>\rho}(1+|\gamma|)^{-s}d\gamma\left(\sum_{i,j=1}^3\|\sigma_{ij}\|^2_{H^s(B_1)}\right)^{\frac12}\\
&\lesssim\|\m\sigma\|_{H^s(B_1;\mathbb R^{3\times3})}\rho^{3-s}. 
\end{align*}
Combining the above estimates leads to 
\begin{eqnarray}\label{eq:sigma_EM}
\Vert \m{\sigma}\Vert_{L^\infty(B_1;\mathbb R^{3\times3})} \lesssim\rho^3\left[k^{m-2}M_{\rm EM}(k)+C_2 k^{-1}+\|\m\sigma\|_{H^s(B_1;\mathbb R^{3\times3})}\rho^{-s}\right].
\end{eqnarray}

The lower bound of $M_{\text{EM}}(k)$ can be obtained similarly to Lemma \ref{lm:M_PL} by utilizing \eqref{eq:E}:
\begin{align}
M_{\rm EM}(k)&=\sup_{\hat x,\hat y\in\mathbb S^2}\left\|\E\left[\m{E}^\infty(\hat{x},k) \m{E}^\infty(\hat{y},k)^\top\right]\right\|_F \nonumber \\
&\ge\left\|\E\left[\m{E}^\infty(\hat{x},k) \m{E}^\infty(-\hat{x},k)^\top\right]\right\|_F\nonumber\\
&=k^2\beta_3^2\left\| k^{-m}\int_{B_1}\m\sigma(z_1)dz_1+\int_{B_1}{\boldsymbol c} (z_1,k\hat x)dz_1\right\|_F\nonumber\\
&\ge\beta_3^2k^{2-m}\left\|\int_{B_1}\m\sigma(z_1)dz_1\right\|_F-\beta_3^2k^2\left\|\int_{B_1}{\boldsymbol c}(z_1,k\hat x)dz_1\right\|_F\nonumber\\
&\ge\beta_3^2k^{2-m}\left(\left\|\int_{B_1}\m\sigma(z_1)dz_1\right\|_F-C_2 k^{-1}\right).\label{eq:estEMk}
\end{align}
Denoting $C_{\sigma}:=\left\|\int_{B_1}\m\sigma(z_1)dz_1\right\|_F>0$ for simplicity, we choose $k_0>\max\{\frac{2C_2}{C_\sigma},1\}\geq 1$. For any $k>k_0$, 
\[
k^{-1}<k_0^{-1}<\frac{C_\sigma}{2C_2},
\]
and together with \eqref{eq:estEMk},
\[
M_{\rm EM}(k)\ge\frac{\beta_3^2C_\sigma}2k^{2-m},
\]
which further implies
\[
k^{-1}k^{2-m}<\frac{C_\sigma}{2C_2}\frac{2M_{\rm EM}(k)}{\beta_3^2C_\sigma}=\frac{M_{\rm EM}(k)}{C_2\beta_3^2}.
\]

Choosing $\rho=k^{\frac1s}<2k$, we obtain from \eqref{eq:sigma_EM} that 
\begin{align*}
\Vert \m{\sigma}\Vert_{L^\infty(B_1;\mathbb R^{3\times3})} &\lesssim k^{\frac3s+m-2}\left[M_{\rm EM}(k)+C_2 k^{-1}k^{2-m}+\|\m\sigma\|_{H^s(B_1;\mathbb R^{3\times3})}k^{-1}k^{2-m}\right]\\
&\lesssim k^{\frac3s+m-2}\left(1+\|\m\sigma\|_{H^s(B_1;\mathbb R^{3\times3})}\right)M_{\rm EM}(k),
\end{align*}
which completes the proof. 
\end{proof}

\section{Elastic  waves}\label{sec:EL}

In this section, we investigate the stability of the inverse random source for elastic waves, which is more complex compared to the cases of polyharmonic and electromagnetic waves.
 
Consider the stochastic elastic wave equation
\begin{equation*}\label{eqn:elas_main}
\Delta^*\m{u} +k^2 \m{u} = \m{f} \quad \text{in } \mathbb{R}^d,
\end{equation*}
where $\Delta^*: = \mu\Delta +(\lambda+\mu)\nabla(\nabla\cdot)$ is the Lam\'{e} operator with Lam\'{e} constants $\mu$ and $\lambda$ satisfying $\mu>0$ and $\lambda+\mu>0$, $k>0$ denotes the wavenumber, and $\m{f}$ is a GMIG random source satisfying Assumption \ref{assump2} with $d\in\{2,3\}$. The goal is to establish the stability estimate of its strength matrix $\m{\sigma}$ by estimating the Fourier coefficients of the entries of $\m\sigma$.

As shown in \cite{bao2020stability}, the displacement vector $\boldsymbol{u}$ can be separated into its compressional component $\boldsymbol{u}_p$ and shear component $\boldsymbol{u}_s$. These components satisfy the Kupradze--Sommerfeld radiation conditions
\begin{equation*}\label{eqn:Kup_Som}
\lim_{r\to\infty} r^{\frac{d-1}{2}} (\partial_r \m{u}_p-\ii k_p \m{u}_p) =0,\quad \lim_{r\to\infty}  r^{\frac{d-1}{2}} (\partial_r  \m{u}_s -\ii k_s  \m{u}_s) =0,\quad r=|x|,
\end{equation*}
where $k_p$ and $k_s$ represent the compressional and shear wavenumbers, respectively, given by
\begin{equation}\label{eq:kps}
k_p = c_pk,\quad k_s = c_sk
\end{equation}
with $c_p = (\lambda+2\mu )^{-\frac12}$ and $c_s = \mu^{-\frac12}$.

According to\cite[Theorem 6.1]{LLW22Far}, the compressional and shear far-field patterns have the following forms:
\begin{align}
\m{u}_p^\infty(\hat{x},k) &= -\Cd c_p^{\frac{d+2}{2}} k^{\frac{d-3}{2}} \hat{x}\hat{x}^\top \int_{\mathbb{R}^d} e^{-\ii k_p \hat{x}\cdot y}\m{f}(y)dy, \label{eqn:farp}\\
\m{u}_s^\infty(\hat{x},k) &= -\Cd c_s^{\frac{d+2}{2}} k^{\frac{d-3}{2}}\left(\m{I}- \hat{x}\hat{x}^\top\right) \int_{\mathbb{R}^d} e^{-\ii k_s \hat{x}\cdot y}\m{f}(y)dy,\label{eqn:fars}
\end{align}
where $\m{I}$ is the $d\times d$ identity matrix. The boundedness of the compressional and shear far-field patterns $\boldsymbol{u}_p^\infty(\hat{x})$ and $ \boldsymbol{u}_s^\infty(\hat{x})$ can be obtained similarly to the polyharmonic wave case, as given in Theorem \ref{tm:u_infty}.

With the measured far-field patterns  $\m{u}^\infty_p$ and $\m{u}^\infty_s$, we define the correlation data
\begin{equation*}\label{eq:defM_elas}
 M_{\rm EL}(k) := \max\left\{ \Vert F_p(\cdot,\cdot,k)\Vert_{L^\infty(\mathbb S^{d-1}\times\mathbb S^{d-1})},\Vert F_s(\cdot,\cdot,k)\Vert_{L^\infty(\mathbb S^{d-1}\times\mathbb S^{d-1})}\right\},
 \end{equation*}
where 
\begin{eqnarray*}
F_p(\hat{x},\hat{y},k)=\left\Vert \E\left[ \m{u}^\infty_p(\hat{x},k) \m{u}^\infty_p(\hat{y},k)^\top\right]\right\Vert_F, \quad F_s(\hat{x},\hat{y},k) =\left\Vert \E\left[ \m{u}^\infty_s(\hat{x},k) \m{u}^\infty_s(\hat{y},k)^\top\right]\right\Vert_F.
\end{eqnarray*}

\begin{lemma}\label{lm:M_EL}
For any $k>1$, the following estimate holds:
\[
M_{\rm EL}(k)\ge\frac{\Cd^2c_p^{d+2-m}}dk^{d-3-m}\|\Tr\m\sigma\|_{L^1(B_1)}-C_2\Cd^2c_p^{d+1-m}k^{d-4-m},
\]
where $\Tr\m\sigma$ denotes the trace of the matrix $\m\sigma$ and $C_2$ is given in \eqref{eq:B}.
\end{lemma}

\begin{proof}
For any $\hat x,\hat y\in\mathbb S^{d-1}$ and $k>1$, it follows from \eqref{eqn:farp} that 
\begin{align}\label{eq:u_p}
&\mathbb E\left[\m{u}^\infty_p(\hat{x},k) \m{u}^\infty_p(\hat{y},k)^\top\right]\notag\\
&=\Cd^2c_p^{d+2}k^{d-3}\hat x\hat x^\top\left[\int_{\mathbb R^d}\int_{\mathbb R^d}e^{-{\rm i}k_p\hat x\cdot z_1-{\rm i}k_p\hat y\cdot z_2}\mathbb E[\m f(z_1)\m f(z_2)]dz_1dz_2\right]\hat y\hat y^\top\notag\\
&=\Cd^2c_p^{d+2}k^{d-3}\hat x\hat x^\top\left[\int_{\mathbb R^d}\int_{\mathbb R^d}e^{-{\rm i}k_p\hat x\cdot z_1-{\rm i}k_p\hat y\cdot z_2}\frac1{(2\pi)^d}\int_{\mathbb R^d}e^{{\rm i}(z_1-z_2)\cdot\xi}{\boldsymbol c}(z_1,\xi)d\xi dz_1dz_2\right]\hat y\hat y^\top\notag\\
&=\Cd^2c_p^{d+2}k^{d-3}\hat x\hat x^\top\left[\int_{\mathbb R^d}\int_{\mathbb R^d}\delta(\xi+k_p\hat y)e^{-{\rm i}k_p\hat x\cdot z_1}e^{{\rm i}z_1\cdot\xi}{\boldsymbol c}(z_1,\xi)d\xi dz_1\right]\hat y\hat y^\top\notag\\
&=\Cd^2c_p^{d+2}k^{d-3}\hat x\hat x^\top\left[\int_{\mathbb R^d}e^{-{\rm i}k_p(\hat x+\hat y)\cdot z_1}\left(\m\sigma(z_1)|k_p|^{-m}+{\boldsymbol r}(z_1,-k_p\hat y)\right)dz_1\right]\hat y\hat y^\top\notag\\
&=\Cd^2c_p^{d+2-m}k^{d-3-m}\hat x\hat x^\top\left[\int_{\mathbb R^d}e^{-{\rm i}k_p(\hat x+\hat y)\cdot z_1}\m\sigma(z_1)dz_1\right]\hat y\hat y^\top\notag\\
&\quad +\Cd^2c_p^{d+2}k^{d-3}\hat x\hat x^\top\left[\int_{\mathbb R^d}e^{-{\rm i}k_p(\hat x+\hat y)\cdot z_1}{\boldsymbol r}(z_1,-k_p\hat y)dz_1\right]\hat y\hat y^\top.
\end{align}
Choosing $\hat y=-e_i$ and $\hat x=e_i$, $i=1,\cdots,d$, where $e_i$ is the unit vector whose $i$th entry is one and other entries are zeroes, we obtain from \eqref{eq:B} that 
\begin{align*}
\Vert F_p(\cdot,\cdot,k)\Vert_{L^\infty(\mathbb S^{d-1}\times\mathbb S^{d-1})}
&\ge \left\|\mathbb E\left[\m{u}^\infty_p(e_i,k) \m{u}^\infty_p(-e_i,k)^\top\right]\right\|_F\\
&\ge\Cd^2c_p^{d+2-m}k^{d-3-m}\left\|e_ie_i^\top\left[\int_{\mathbb R^d}\m\sigma(z_1)dz_1\right]e_ie_i^\top\right\|_F\\
&\quad -\Cd^2c_p^{d+2}k^{d-3}\left\|e_ie_i^\top\left[\int_{\mathbb R^d}{\boldsymbol r}(z_1,-k_p\hat y)dz_1\right]e_ie_i^\top\right\|_F\\
&=\Cd^2c_p^{d+2-m}k^{d-3-m}\left|\int_{\mathbb R^d}\sigma_{ii}(z_1)dz_1\right|-\Cd^2c_p^{d+2}k^{d-3}\left|\int_{\mathbb R^d}r_{ii}(z_1,-k_p\hat y)dz_1\right|\\
&\ge \Cd^2c_p^{d+2-m}k^{d-3-m}\int_{\mathbb R^d}\sigma_{ii}(z_1)dz_1-C_2\Cd^2c_p^{d+1-m}k^{d-4-m}, 
\end{align*}
where in the last step we utilize the fact that $\sigma_{ii}\ge0$ for any $i=1,\cdots,d,$ since $\boldsymbol{\sigma}$ is non-negative definite. Combining the above estimates for $i=1,\cdots,d$, we obtain
\[
\Vert F_p(\cdot,\cdot,k)\Vert_{L^\infty(\mathbb S^{d-1}\times\mathbb S^{d-1})}
\ge\frac{\Cd^2c_p^{d+2-m}}dk^{d-3-m}\|\Tr\m\sigma\|_{L^1(B_1)}-C_2\Cd^2c_p^{d+1-m}k^{d-4-m},
\]
which completes the proof.
\end{proof}

As stated at the beginning of this section, the stability estimate for the reconstruction of the strength $\m\sigma$ for elastic waves is more challenging than those for the polyharmonic and electromagnetic waves presented in Sections \ref{sec:PL} and \ref{sec:EM}. The reason is that the matrices $\hat x\hat x^\top$ and $\m I-\hat x\hat x^\top$ involved in the far-field patterns \eqref{eqn:farp}--\eqref{eqn:fars} are not invertible, which makes it  different from Lemmas \ref{lm:Fourier_PL} and \ref{lm:Fourier_EM} and rather difficult to get the estimate for the Fourier coefficient matrix $\widehat{\m\sigma}(\gamma)$.

\begin{lemma}\label{lm:Fourier_EL}
Let $\m f$ satisfy Assumption \ref{assump2} with $d\in\{2,3\}$. Assume in addition that $\m{\sigma}$ is a symmetric matrix. For any $k>1$ and $\gamma\in\mathbb R^d$ with $|\gamma|\leq k_p$, it holds 
\[
\left|\widehat{\Tr{\m\sigma}}(\gamma)\right|\le \frac{2dc_p^{m-d-2}}{\Cd^2}k^{m-d+3}M_{\rm EL}(k)+2dC_2 c_p^{-1}k^{-1}
\]
where $C_2$ is given in \eqref{eq:B}.
\end{lemma}

\begin{proof}
The proof is divided into two parts, each examining the correlation data associated with $\boldsymbol{u}_p^\infty$ and $\boldsymbol{u}_s^\infty$, respectively.

First we consider the correlation data of the compressional far-field patterns. For any $\gamma\in\mathbb R^d$ with $|\gamma|\le k_p$, there exist unit vectors $\nu_1,\cdots,\nu_{d-1}\in\mathbb S^{d-1}$ such that $\gamma\cdot\nu_i=\nu_i\cdot\nu_j=0$ for $i=1,\cdots,d-1,$ and $i\neq j$. Define two unit vectors 
\[
\hat{x}_p := \frac{\gamma + (4k_p^2-|\gamma|^2)^{1/2}\nu_1}{2k_p}\in\mathbb S^{d-1},\quad \hat{y}_p: = \frac{\gamma - (4k_p^2-|\gamma|^2)^{1/2}\nu_1}{2k_p}\in\mathbb S^{d-1}
\]
such that $\gamma=k_p(\hat x_p+\hat y_p)$. It follows from \eqref{eq:u_p} that
\begin{align*}
\hat x_p^\top\mathbb E\left[\m{u}^\infty_p(\hat{x}_p,k) \m{u}^\infty_p(\hat{y}_p,k)^\top\right]\hat y_p
&=\Cd^2c_p^{d+2-m}k^{d-3-m}\hat x^\top_p\widehat{\m\sigma}(\gamma)\hat y_p\\
&\quad +\Cd^2c_p^{d+2}k^{d-3}\hat x^\top_p\left[\int_{\mathbb R^d}e^{-{\rm i}\gamma\cdot z_1}{\boldsymbol r}(z_1,-k_p\hat y_p)dz_1\right]\hat y_p,
\end{align*}
which, together with  \eqref{eq:B}, implies 
\begin{align}\label{eq:sigma_p1}
\left|\hat x^\top_p\widehat{\m\sigma}(\gamma)\hat y_p\right|
&\le \frac{c_p^{m-d-2}}{\Cd^2}k^{m-d+3}\left|\hat x_p^\top\mathbb E\left[\m{u}^\infty_p(\hat{x}_p,k) \m{u}^\infty_p(\hat{y}_p,k)^\top\right]\hat y_p\right|\notag\\
&\quad +c_p^mk^m\left|\hat x_p^\top\left[\int_{\mathbb R^d}e^{-{\rm i}\gamma\cdot z_1}{\boldsymbol r}(z_1,-k_p\hat y_p)dz_1\right]\hat y_p\right|\notag\\
&\le \frac{c_p^{m-d-2}}{\Cd^2}k^{m-d+3}\left\|\mathbb E\left[\m{u}^\infty_p(\hat{x}_p,k) \m{u}^\infty_p(\hat{y}_p,k)^\top\right]\right\|_F+c_p^mk^m\left\|\int_{\mathbb R^d}e^{-{\rm i}\gamma\cdot z_1}{\boldsymbol r}(z_1,-k_p\hat y_p)dz_1\right\|_F\notag\\
&\le\frac{c_p^{m-d-2}}{\Cd^2}k^{m-d+3}M_{\rm EL}(k)+C_2 c_p^{-1}k^{-1}. 
\end{align}

When $\gamma=0$, $\nu_1$ can be chosen as any unit vector, and we have $\hat{x}_p=\nu_1$ and $\hat{y}_p=-\nu_1$. Then by choosing $\nu_1$ as the canonical basis vectors $e_i$, $i=1,\cdots,d$, the above estimate yields
\begin{align*}
\left|\widehat{\Tr{\m\sigma}}(0)\right|&\le \sum_{i=1}^d|\widehat{\sigma_{ii}}(0)|=\sum_{i=1}^d\left|e_i^\top\widehat{\m\sigma}(0)e_i\right|\\
&\le\frac{dc_p^{m-d-2}}{\Cd^2}k^{m-d+3}M_{\rm EL}(k)+dC_2 c_p^{-1}k^{-1},
\end{align*}
which satisfies the result. 

It suffices to consider the estimate for $\gamma\neq0$.
For simplicity, we denote $\theta_p:=\frac{|\gamma|}{2k_p}\in(0,\frac12]$, the unitary matrix $\boldsymbol U:=\left[\begin{matrix}\frac{\gamma}{|\gamma|}&\nu_1&\cdots&\nu_{d-1}\end{matrix}\right]_{d\times d}$, and $\boldsymbol B=[b_{ij}]_{d\times d}:=\boldsymbol U^\top\widehat{\m\sigma}(\gamma)\boldsymbol U$. Noting that the matrix $\boldsymbol B$ is symmetric due to the symmetry of $\m\sigma$, and
\[
\hat x_p=\boldsymbol U(\theta_p e_1+(1-\theta_p^2)^{1/2}e_2 ),\quad
\hat y_p=\boldsymbol U(\theta_p e_1-(1-\theta_p^2)^{1/2}e_2),
\]
we obtain 
\begin{eqnarray}\label{eq:sigma_p2}
\hat x^\top_p\widehat{\m\sigma}(\gamma)\hat y_p
=(\theta_p e_1^\top+(1-\theta_p^2)^{1/2}e_2^\top )\boldsymbol B(\theta_p e_1+(1-\theta_p^2)^{1/2}e_2 )=\theta_p^2b_{11}-(1-\theta_p^2)b_{22}.
\end{eqnarray}

Next we consider the correlation data of the shear far-field patterns. For the given $\gamma\neq0$, it holds that $|\gamma|\le k_p< k_s$ according to \eqref{eq:kps}. Therefore, there also exist
\[
\hat{x}_s := \frac{\gamma + (4k_s^2-|\gamma|^2)^{1/2}\nu_1}{2k_s}\in\mathbb S^{d-1},\quad \hat{y}_s: = \frac{\gamma - (4k_s^2-|\gamma|^2)^{1/2}\nu_1}{2k_s}\in\mathbb S^{d-1}
\]
such that $\gamma=k_s(\hat x_s+\hat y_s)$. We define two additional unit vectors 
\[
\hat\rho_1 := \frac{(4k_s^2-|\gamma|^2)^{1/2}\gamma-|\gamma|^2\nu_1}{2k_s|\gamma|}\in\mathbb S^{d-1}, \quad \hat\rho_2 :=  \frac{(4k_s^2-|\gamma|^2)^{1/2}\gamma+|\gamma|^2\nu_1}{2k_s|\gamma|}\in\mathbb S^{d-1},
\]
which satisfy $\hat x\cdot\hat\rho_1=\hat y\cdot\hat\rho_2=0$. Similar to \eqref{eq:u_p}, we consider the correlation data for $\m u_s^{\infty}$ and obtain 
\begin{align}\label{eq:u_s}
&\mathbb E\left[\m{u}^\infty_s(\hat{x}_s,k) \m{u}^\infty_s(\hat{y}_s,k)^\top\right]\notag\\
&=\Cd^2c_s^{d+2}k^{d-3}\left(\m I-\hat x_s\hat x_s^\top\right)\left[\int_{\mathbb R^d}\int_{\mathbb R^d}e^{-{\rm i}k_s\hat x_s\cdot z_1-{\rm i}k_s\hat y_s\cdot z_2}\mathbb E[\m f(z_1)\m f(z_2)]dz_1dz_2\right]\left(\m I-\hat y_s\hat y_s^\top\right)\notag\\
&=\Cd^2c_s^{d+2-m}k^{d-3-m}\left(\m I-\hat x_s\hat x_s^\top\right)\left[\int_{\mathbb R^d}e^{-{\rm i}\gamma\cdot z_1}\m\sigma(z_1)dz_1\right]\left(\m I-\hat y_s\hat y_s^\top\right)\notag\\
&\quad +\Cd^2c_s^{d+2}k^{d-3}\left(\m I-\hat x_s\hat x_s^\top\right)\left[\int_{\mathbb R^d}e^{-{\rm i}\gamma\cdot z_1}{\boldsymbol r}(z_1,-k_s\hat y_s)dz_1\right]\left(\m I-\hat y_s\hat y_s^\top\right).
\end{align}
Multiplying the above equation by $\hat\rho_1^\top$ on the left and by $\hat\rho_2$ on the right, we have 
\begin{align}\label{eq:sigma_s1}
\left|\hat\rho_1^\top\widehat{\m\sigma}(\gamma)\hat\rho_2\right|
&\le \frac{c_s^{m-d-2}}{\Cd^2}k^{m-d+3}\left|\hat\rho_1^\top\mathbb E\left[\m{u}^\infty_s(\hat{x}_s,k) \m{u}^\infty_s(\hat{y}_s,k)^\top\right]\hat\rho_2\right|\notag\\
&\quad +c_s^mk^m\left|\hat\rho_1^\top\left[\int_{\mathbb R^d}e^{-{\rm i}\gamma\cdot z_1}{\boldsymbol r}(z_1,-k_s\hat y_s)dz_1\right]\hat\rho_2\right|\notag\\
&\le \frac{c_s^{m-d-2}}{\Cd^2}k^{m-d+3}M_{\rm EL}(k)+C_2 c_s^{-1}k^{-1}.
\end{align}

Similarly, we denote $\theta_s:=\frac{|\gamma|}{2k_s}\in(0,\frac{c_p}{2c_s}]$ with $\frac{c_p}{2c_s}<\frac12$ and get
\begin{eqnarray}\label{eq:sigma_s2}
\hat\rho_1^\top\widehat{\m\sigma}(\gamma)\hat\rho_2
=(1-\theta_s^2)b_{11}-\theta_s^2b_{22}.
\end{eqnarray}
Combining \eqref{eq:sigma_p2} and \eqref{eq:sigma_s2} gives
\[
\left[\begin{matrix}
\theta_p^2&-(1-\theta_p^2)\\[5pt]
1-\theta_s^2&-\theta_s^2
\end{matrix}\right]
\left[\begin{matrix}
b_{11}\\[5pt]
b_{22}
\end{matrix}\right]
=\left[\begin{matrix}
\hat x^\top_p\widehat{\m\sigma}(\gamma)\hat y_p\\[5pt]
\hat\rho_1^\top\widehat{\m\sigma}(\gamma)\hat\rho_2
\end{matrix}\right],
\]
where the coefficient matrix
\[
\Theta:=\left[\begin{matrix}
\theta_p^2&-(1-\theta_p^2)\\
1-\theta_s^2&-\theta_s^2
\end{matrix}\right]
\]
is invertible, satisfying 
\[
\det\Theta=1-\theta_p^2-\theta_s^2 \geq 1-\left(\frac12\right)^2-\left(\frac{c_p}{2c_s}\right)^2=\frac34-\frac{\mu}{4(\lambda+2\mu)}>\frac12.
\]
Hence,
\begin{eqnarray*}
\left[\begin{matrix}
b_{11}\\[5pt]
b_{22}
\end{matrix}\right]
=\Theta^{-1}\left[\begin{matrix}
\hat x^\top_p\widehat{\m\sigma}(\gamma)\hat y_p\\[5pt]
\hat\rho_1^\top\widehat{\m\sigma}(\gamma)\hat\rho_2
\end{matrix}\right]=\frac1{\det\Theta}
\left[\begin{matrix}
-\theta_s^2&-(1-\theta_s^2)\\[5pt]
1-\theta_p^2&\theta_p^2
\end{matrix}\right]
\left[\begin{matrix}
\hat x^\top_p\widehat{\m\sigma}(\gamma)\hat y_p\\[5pt]
\hat\rho_1^\top\widehat{\m\sigma}(\gamma)\hat\rho_2
\end{matrix}\right],
\end{eqnarray*}
which, together with \eqref{eq:sigma_p1}, \eqref{eq:sigma_s1}, and the fact  that $c_s>c_p$, leads to
\begin{align}\label{eq:b12}
\max_{i=1,2}\{|b_{ii}|\}
&\le \frac1{\det\Theta}\max\left\{|\hat x^\top_p\widehat{\m\sigma}(\gamma)\hat y_p|,\,|\hat\rho_1^\top\widehat{\m\sigma}(\gamma)\hat\rho_2|\right\}\notag\\
&\le 2\left[\frac{c_p^{m-d-2}}{\Cd^2}k^{m-d+3}M_{\rm EL}(k)+C_2 c_p^{-1}k^{-1}\right]. 
\end{align}

For $d=2$, we conclude from \eqref{eq:b12} that
\begin{align*}
\left|\widehat{\Tr{\m\sigma}}(\gamma)\right|&=\left|\Tr[\widehat{\m\sigma}(\gamma)]\right|=|\Tr \boldsymbol B|\le d\max_{i=1,2}\{|b_{ii}|\}\\
&\le \frac{2dc_p^{m-4}}{\beta_2^2}k^{m+1}M_{\rm EL}(k)+2dC_2 c_p^{-1}k^{-1},\quad |\gamma|\le k_p.
\end{align*}
For $d=3$, we need to estimate $b_{33}$. Multiplying \eqref{eq:u_s} by $\nu_2^\top$ on the left and by $\nu_2$ on the right, we get
\begin{align*}
\left|\nu_2^\top\widehat{\m\sigma}(\gamma)\nu_2\right|
&=\left|e_3^\top \boldsymbol B e_3\right|=|b_{33}|\notag\\
&\le \frac{c_s^{m-d-2}}{\Cd^2}k^{m-d+3}\left|\nu_2^\top\mathbb E\left[\m{u}^\infty_s(\hat{x}_s,k) \m{u}^\infty_s(\hat{y}_s,k)^\top\right]\nu_2\right|\\
&\quad +c_s^mk^m\left|\nu_2^\top\left[\int_{\mathbb R^d}e^{-{\rm i}\gamma\cdot z_1}{\boldsymbol r}(z_1,-k_s\hat y_s)dz_1\right]\nu_2\right|\\
&\le\frac{c_s^{m-d-2}}{\Cd^2}k^{m-d+3}M_{\rm EL}(k)+C_2 c_s^{-1}k^{-1},
\end{align*}
which, together with \eqref{eq:b12}, yields
\begin{align*}
\left|\widehat{\Tr{\m\sigma}}(\gamma)\right|&=|\Tr \boldsymbol B|\le d\max_{i=1,2,3}\{|b_{ii}|\}\\
&\le \frac{2dc_p^{m-4}}{\beta_2^2}k^{m+1}M_{\rm EL}(k)+2dC_2 c_p^{-1}k^{-1},\quad|\gamma|\le k_p,
\end{align*}
thus completing the proof. 
\end{proof}

Following a similar argument as the proof of Theorem \ref{stab_poly}, we derive the following stability estimate for $\Tr(\m{\sigma})$.

\begin{theorem}
Let $\m f$ satisfy Assumption \ref{assump2} with $d\in\{2,3\}$ and additionally $s>d$. Assume also that $\m{\sigma}$ is a symmetric matrix. There exists a constant $k_0>1$ such that for any $k>k_0$, the following estimate holds:
\[
\left\|\Tr{\m\sigma}\right\|_{L^\infty(B_1)}\lesssim k^{\frac ds+m-d+3}\left(1+\|\Tr{\m\sigma}\|_{H^s(B_1)}\right)M_{\rm EL}(k).
\]
\end{theorem}

\begin{proof}
Note that
\begin{align*}
\Vert \Tr\m{\sigma}\Vert_{L^\infty(B_1)} 
&=\sup_{x\in B_1}\left|\frac1{(2\pi)^d}\int_{\mathbb R^d}\widehat{\Tr\m\sigma}(\gamma)e^{-{\rm i}x\cdot\gamma}d\gamma\right|\\
&\le \frac1{(2\pi)^d}\int_{|\gamma|\leq \rho}\left|\widehat{\Tr\m{\sigma}}(\gamma)\right|d\gamma+ \frac1{(2\pi)^d}\int_{|\gamma|>\rho} \left|\widehat{\Tr\m{\sigma}}(\gamma)\right|d\gamma\\ 
&=: I_1(\rho)+I_2(\rho),
\end{align*}
where
\[
I_1(\rho)\le\frac{\lambda(B_\rho)}{(2\pi)^d}\left[\frac{2dc_p^{m-d-2}}{\Cd^2}k^{m-d+3}M_{\rm EL}(k)+2dC_2 c_p^{-1}k^{-1}\right]
\]
for $\rho\le k_p$ according to Lemma \ref{lm:Fourier_EL}, and
\begin{eqnarray*}
I_2(\rho)
\lesssim\int_{|\gamma|>\rho}(1+|\gamma|)^{-s}d\gamma\|\Tr\m\sigma\|_{H^s(B_1)}
\lesssim\|\Tr\m\sigma\|_{H^s(B_1)}\rho^{d-s}
\end{eqnarray*}
for $s>d$ according to the Paley--Wiener--Schwartz theorem used in the proof of Theorem \ref{stab_poly}.
We then conclude
\begin{eqnarray}\label{eq:sigma_EL}
\|\Tr\m\sigma\|_{L^\infty(B_1)}\lesssim\rho^d\left[k^{m-d+3}M_{\rm EL}(k)+C_2 k^{-1}+\|\Tr\m\sigma\|_{H^s(B_1)}\rho^{-s}\right].
\end{eqnarray}

Let $k_0>\max
\Big\{\frac{2dC_2}{c_p\|\Tr\m\sigma\|_{L^1(B_1)}},1\Big\}\geq1$. For any $k>k_0$,
\[
k^{-1}<k_0^{-1}<\frac{c_p\|\Tr\m\sigma\|_{L^1(B_1)}}{2dC_2},
\]
which, together with  the lower bound of the data $M_{\rm EL}(k)$ in Lemma \ref{lm:M_EL}, leads to
\[
M_{\rm EL}(k)\ge\frac{\Cd^2c_p^{d+2-m}\|\Tr\m\sigma\|_{L^1(B_1)}}{2d}k^{d-3-m}.
\]
Thus we obtain that
\[
k^{-1}k^{d-3-m}<\frac{c_p\|\Tr\m\sigma\|_{L^1(B_1)}}{2dC_2}\frac{2dM_{\rm EL}(k)}{\Cd^2c_p^{d+2-m}\|\Tr\m\sigma\|_{L^1(B_1)}}=\frac{M_{\rm EL}(k)}{\Cd^2c_p^{d+1-m}C_2}.
\]
By choosing $\rho=k_p^{\frac1s}<k_p$, it follows from \eqref{eq:sigma_EL} that
\begin{align*}
\Vert \m{\sigma}\Vert_{L^\infty(B_1;\mathbb R^{3\times3})} &\lesssim k^{\frac ds+m-d+3}\left[M_{\rm EL}(k)+C_2 k^{-1}k^{d-3-m}+\|\Tr\m\sigma\|_{H^s(B_1)}k^{-1}k^{d-3-m}\right]\\
&\lesssim k^{\frac ds+m-d+3}\left(1+\|\Tr\m\sigma\|_{H^s(B_1)}\right)M_{\rm EL}(k),
\end{align*}
which completes the proof.
\end{proof}

\section{Conclusion}\label{con}

In this paper, we have examined the stability of inverse random source problems concerning polyharmonic, electromagnetic, and elastic waves within a unified framework. The source is characterized as a real-valued centered GMIG random field, where the covariance operator is represented as a classical pseudo-differential operator. The objective of the inverse source problem is to identify the strength function corresponding to the principal symbol of the covariance operator by employing the correlation of the far-field patterns. We establish Lipschitz stability across all cases using single-frequency data, illustrating that the inverse random source problems are stable when utilizing correlation-based data.

A possible extension of this research involves exploring the stochastic wave equation in a random medium, where the medium coefficient is modeled as a GMIG random field. The challenges arise from the presence of nonlinearity and multiplicative noise associated with the inverse random medium problem. Our objective is to advance this line of research and provide deeper insights into these more intricate problems in future studies.

\section*{Appendix}

\begin{proof}[Proof of Lemma \ref{lm:H}]
For any smooth test functions $\phi\in C_0^\infty(D)$ and $\psi\in C_0^\infty(B)$, we retain the notation of $\phi$ and $\psi$ as their respective zero extensions outside of $D$ and $B$. It follows from the Parseval--Plancherel identity that
\begin{align*}
\left\langle\mathcal H_k\phi,\psi\right\rangle &=\int_{\mathbb R^d}\frac1{|\xi|^{2n}-k^{2n}}\widehat\phi(\xi)\widehat\psi(\xi)d\xi\notag\\
&=\int_{\mathbb R^d}\frac{(1+|\xi|^2)^{\frac s2}}{|\xi|^{2n}-k^{2n}}\widehat{\mathcal J^{-s_1}\phi}(\xi)\widehat{\mathcal J^{-s_2}\psi}(\xi)d\xi,
\end{align*}
where 
\[
\left(\mathcal J^{-s_1}\phi\right)(x):=\mathcal F^{-1}\left[(1+|\cdot|^2)^{-\frac{s_1}2}\widehat\phi\right](x)
\]
denotes the Bessel potential of order $-s_1$, with $\mathcal F^{-1}$ representing the inverse Fourier transform. Divide the entire domain $\mathbb R^d$ into two subdomains
\[
\Omega_1:=\left\{\xi\in\mathbb R^d:||\xi|-k|>\frac k2\right\},\quad \Omega_2:=\left\{\xi\in\mathbb R^d:||\xi|-k|<\frac k2\right\}
\]
such that
\begin{align*}
\left\langle\mathcal H_k\phi,\psi\right\rangle& =\int_{\Omega_1}\frac{(1+|\xi|^2)^{\frac s2}}{|\xi|^{2n}-k^{2n}}\widehat{\mathcal J^{-s_1}\phi}(\xi)\widehat{\mathcal J^{-s_2}\psi}(\xi)d\xi+\int_{\Omega_2}\frac{(1+|\xi|^2)^{\frac s2}}{|\xi|^{2n}-k^{2n}}\widehat{\mathcal J^{-s_1}\phi}(\xi)\widehat{\mathcal J^{-s_2}\psi}(\xi)d\xi\\
&=:I+K
\end{align*} 
and only the second term $K$ is a singular integral. 

The first term $I$ can be estimated easily
\begin{align*}
I &=\int_{\Omega_1}\frac{(1+|\xi|^2)^{\frac s2}}{(|\xi|-k)(|\xi|+k)\left(\sum_{j=0}^{n-1}|\xi|^{2j}k^{2(n-1-j)}\right)}\widehat{\mathcal J^{-s_1}\phi}(\xi)\widehat{\mathcal J^{-s_2}\psi}(\xi)d\xi\\
&\lesssim \frac1k\int_{\{|\xi|>\frac{3k}2\}\cup\{|\xi|<\frac k2\}}\frac{(1+|\xi|^2)^{\frac s2}}{(|\xi|+k)\left(\sum_{j=0}^{n-1}|\xi|^{2j}k^{2(n-1-j)}\right)}\left|\widehat{\mathcal J^{-s_1}\phi}\widehat{\mathcal J^{-s_2}\psi}\right|d\xi\\
&\lesssim \frac1k\int_{\{|\xi|>\frac{3k}2\}}\frac1{|\xi|^{2n-1-s}}\left|\widehat{\mathcal J^{-s_1}\phi}\widehat{\mathcal J^{-s_2}\psi}\right|d\xi
+\frac1k\int_{\{|\xi|<\frac k2\}}\frac{1+k^s}{k^{2n-1}}\left|\widehat{\mathcal J^{-s_1}\phi}\widehat{\mathcal J^{-s_2}\psi}\right|d\xi\\
&\lesssim k^{s-2n}\|\phi\|_{H^{-s_1}(D)}\|\psi\|_{H^{-s_2}(B)}.
\end{align*}

For the singular integral $K$, we adopt the coordinate transformation $\mathcal T:\xi\mapsto\xi^*:=\left(\frac{2k}{|\xi|}-1\right)\xi$ in $\Omega_{2}$, where $|\xi^*|=2k-|\xi|$. This transformation maps the subdomain 
\[
\Omega_{21}:=\left\{\xi\in\Omega_2:\frac k2<|\xi|<k\right\}
\]
to the subdomain
\[
\Omega_{22}:=\left\{\xi\in\Omega_2:k<|\xi|<\frac{3k}2\right\}
\]
with the Jacobian $J_d(\xi)=\left(\frac{2k}{|\xi|}-1\right)^{d-1}$. Following a procedure similar to that described in \cite[Lemma 2.2]{LW24}, we decompose the term $K$ into four parts:
\begin{align*}
K &=\int_{\Omega_{21}}\frac{(1+|\xi|^2)^{\frac s2}}{|\xi|^{2n}-k^{2n}}\widehat{\mathcal J^{-s_1}\phi}(\xi)\widehat{\mathcal J^{-s_2}\psi}(\xi)d\xi+\int_{\Omega_{22}}\frac{(1+|\xi^*|^2)^{\frac s2}}{|\xi^*|^{2n}-k^{2n}}\widehat{\mathcal J^{-s_1}\phi}(\xi^*)\widehat{\mathcal J^{-s_2}\psi}(\xi^*)d\xi^*\\
&=\int_{\Omega_{21}}\frac{(1+|\xi|^2)^{\frac s2}}{|\xi|^{2n}-k^{2n}}\widehat{\mathcal J^{-s_1}\phi}(\xi)\widehat{\mathcal J^{-s_2}\psi}(\xi)d\xi+\int_{\Omega_{21}}\frac{(1+|\xi^*|^2)^{\frac s2}}{|\xi^*|^{2n}-k^{2n}}\widehat{\mathcal J^{-s_1}\phi}(\xi^*)\widehat{\mathcal J^{-s_2}\psi}(\xi^*)J_d(\xi)d\xi\\
&=\int_{\Omega_{21}}\left[\frac{1}{|\xi|^{2n}-k^{2n}}+\frac{J_d(\xi)}{|\xi^*|^{2n}-k^{2n}}\right](1+|\xi|^2)^{\frac s2}\widehat{\mathcal J^{-s_1}\phi}(\xi)\widehat{\mathcal J^{-s_2}\psi}(\xi)d\xi\\
&\quad +\int_{\Omega_{21}}\frac{J_d(\xi)}{|\xi^*|^{2n}-k^{2n}}\left[(1+|\xi^*|^2)^{\frac s2}-(1+|\xi|^2)^{\frac s2}\right]\widehat{\mathcal J^{-s_1}\phi}(\xi)\widehat{\mathcal J^{-s_2}\psi}(\xi)d\xi\\
&\quad +\int_{\Omega_{21}}\frac{J_d(\xi)}{|\xi^*|^{2n}-k^{2n}}(1+|\xi^*|^2)^{\frac s2}\left[\widehat{\mathcal J^{-s_1}\phi}(\xi^*)-\widehat{\mathcal J^{-s_1}\phi}(\xi)\right]\widehat{\mathcal J^{-s_2}\psi}(\xi)d\xi\\
&\quad +\int_{\Omega_{21}}\frac{J_d(\xi)}{|\xi^*|^{2n}-k^{2n}}(1+|\xi^*|^2)^{\frac s2}\widehat{\mathcal J^{-s_1}\phi}(\xi^*)\left[\widehat{\mathcal J^{-s_2}\psi}(\xi^*)-\widehat{\mathcal J^{-s_2}\psi}(\xi)\right]d\xi\\
&=:K_1+K_2+K_3+K_4.
\end{align*}

Define the function
\begin{align*}
F(\xi):=\frac{1}{|\xi|^{2n}-k^{2n}}+\frac{J_d(\xi)}{|\xi^*|^{2n}-k^{2n}},\quad \xi\in\Omega_{21}.
\end{align*}
For $d=2$, it holds uniformly for $\xi\in\Omega_{21}$ that
\begin{align*}
F(\xi)&=\frac1{(|\xi|-k)(|\xi|+k)\left(\sum_{j=0}^{n-1}|\xi|^{2j}k^{2(n-1-j)}\right)}\\
&\quad +\frac{2k-|\xi|}{|\xi|(k-|\xi|)(3k-|\xi|)\left(\sum_{j=0}^{n-1}(2k-|\xi|)^{2j}k^{2(n-1-j)}\right)}\\
&=\frac{|\xi|(|\xi|-3k)\left(\sum_{j=0}^{n-1}(2k-|\xi|)^{2j}k^{2(n-1-j)}\right)+(2k-|\xi|)(|\xi|+k)\left(\sum_{j=0}^{n-1}|\xi|^{2j}k^{2(n-1-j)}\right)}{|\xi|(|\xi|-k)(|\xi|+k)(|\xi|-3k)\left(\sum_{j=0}^{n-1}|\xi|^{2j}k^{2(n-1-j)}\right)\left(\sum_{j=0}^{n-1}(2k-|\xi|)^{2j}k^{2(n-1-j)}\right)}\\
&=\frac{\sum_{j=0}^{n-1}k^{2(n-1-j)}\left\{(|\xi|^2-3k|\xi|)\left[(2k-|\xi|)^{2j}-|\xi|^{2j}\right]+2k(k-|\xi|)|\xi|^{2j}\right\}}{|\xi|(|\xi|-k)(|\xi|+k)(|\xi|-3k)\left(\sum_{j=0}^{n-1}|\xi|^{2j}k^{2(n-1-j)}\right)\left(\sum_{j=0}^{n-1}(2k-|\xi|)^{2j}k^{2(n-1-j)}\right)}\\
&=-\frac{\sum_{j=0}^{n-1}k^{2(n-1-j)}\left\{(|\xi|^2-3k|\xi|)4k\left[\sum_{l=0}^{j-1}(2k-|\xi|)^{2l}|\xi|^{2(j-1-l)}\right]+2k|\xi|^{2j}\right\}}{|\xi|(|\xi|+k)(|\xi|-3k)\left(\sum_{j=0}^{n-1}|\xi|^{2j}k^{2(n-1-j)}\right)\left(\sum_{j=0}^{n-1}(2k-|\xi|)^{2j}k^{2(n-1-j)}\right)}\\
&\lesssim k^{-2n}.
\end{align*}
For $d=3$, it holds similarly that
\begin{align*}
F(\xi)&=\frac1{(|\xi|-k)(|\xi|+k)\left(\sum_{j=0}^{n-1}|\xi|^{2j}k^{2(n-1-j)}\right)}\\
&\quad +\frac{(2k-|\xi|)^2}{|\xi|^2(k-|\xi|)(3k-|\xi|)\left(\sum_{j=0}^{n-1}(2k-|\xi|)^{2j}k^{2(n-1-j)}\right)}\\
&=\frac{\sum_{j=0}^{n-1}k^{2(n-1-j)}\left\{(|\xi|^3-3k|\xi|^2)\left[(2k-|\xi|)^{2j}-|\xi|^{2j}\right]+2(|\xi|^2-2k|\xi|-2k^2)(|\xi|-k)|\xi|^{2j}\right\}}{|\xi|^2(|\xi|-k)(|\xi|+k)(|\xi|-3k)\left(\sum_{j=0}^{n-1}|\xi|^{2j}k^{2(n-1-j)}\right)\left(\sum_{j=0}^{n-1}(2k-|\xi|)^{2j}k^{2(n-1-j)}\right)}\\
&=\frac{\sum_{j=0}^{n-1}k^{2(n-1-j)}\left\{-(|\xi|^3-3k|\xi|^2)4k\left[\sum_{l=0}^{j-1}(2k-|\xi|)^{2l}|\xi|^{2(j-1-l)}\right]+2(|\xi|^2-2k|\xi|-2k^2)|\xi|^{2j}\right\}}{|\xi|^2(|\xi|+k)(|\xi|-3k)\left(\sum_{j=0}^{n-1}|\xi|^{2j}k^{2(n-1-j)}\right)\left(\sum_{j=0}^{n-1}(2k-|\xi|)^{2j}k^{2(n-1-j)}\right)}\\
&\lesssim k^{-2n}.
\end{align*}
Combining the above estimates leads to
\[
K_1\lesssim k^{s-2n}\|\phi\|_{H^{-s_1}(D)}\|\psi\|_{H^{-s_2}(B)}.
\]

It is evident that the term $K_2$ satisfies
\begin{align*}
K_2&=\int_{\Omega_{21}}\frac{J_d(\xi)}{|\xi^*|^{2n}-k^{2n}}\left[(1+|\xi^*|^2)^{\frac s2}-(1+|\xi|^2)^{\frac s2}\right]\widehat{\mathcal J^{-s_1}\phi}(\xi)\widehat{\mathcal J^{-s_2}\psi}(\xi)d\xi\\
&=\int_{\Omega_{21}}\frac{J_d(\xi)(|\xi^*|^2-|\xi|^2)}{|\xi^*|^{2n}-k^{2n}}\frac s2\left(1+\theta|\xi^*|^2+(1-\theta)|\xi|^2\right)^{\frac s2-1}\widehat{\mathcal J^{-s_1}\phi}(\xi)\widehat{\mathcal J^{-s_2}\psi}(\xi)d\xi\\
&=\int_{\Omega_{21}}\frac{s(2k-|\xi|)^{d-2}(|\xi^*|+|\xi|)}{|\xi|^{d-2}(|\xi^*|+k)\sum_{j=0}^{n-1}|\xi^*|^{2j}k^{2(n-1-j)}}\left(1+\theta|\xi^*|^2+(1-\theta)|\xi|^2\right)^{\frac s2-1}\widehat{\mathcal J^{-s_1}\phi}(\xi)\widehat{\mathcal J^{-s_2}\psi}(\xi)d\xi\\
&\lesssim k^{s-2n}\|\phi\|_{H^{-s_1}(D)}\|\psi\|_{H^{-s_2}(B)},
\end{align*}
where $\theta\in(0,1)$ is a constant.

For the estimates of $K_3$ and $K_4$, we note that $\widehat{\mathcal J^{-s_1}\phi},\widehat{\mathcal J^{-s_2}\psi}\in\mathcal S(\mathbb R^d)\subset H^1(\mathbb R^d)$ for $\phi,\psi\in C_0^{\infty}(\mathbb R^d)$. The characterization of $H^1(\mathbb R^d)$ given in \cite{H96b,P04} indicates that 
\[
\left|\widehat{\mathcal J^{-s_1}\phi}(\xi^*)-\widehat{\mathcal J^{-s_1}\phi}(\xi)\right|\lesssim|\xi^*-\xi|\left[M(|\nabla\widehat{\mathcal J^{-s_1}\phi}|)(\xi^*)+M(|\nabla\widehat{\mathcal J^{-s_1}\phi}|)(\xi)\right],
\]
where $M(\cdot)$ denotes the Hardy--Littlewood maximal function satisfying (cf. \cite{P04})
\[
\|M(|\nabla\widehat{\mathcal J^{-s_1}\phi}|)\|_{L^2(\mathbb R^d)}\lesssim\|\nabla\widehat{\mathcal J^{-s_1}\phi}\|_{L^2(\mathbb R^d)}\lesssim\|\phi\|_{H^{-s_1}(D)}.
\]
Similar results apply to $\widehat{\mathcal J^{-s_2}\psi}$. Using the estimate
\begin{align*}
\left|\frac{J_d(\xi)}{|\xi^*|^{2n}-k^{2n}}(1+|\xi^*|^2)^{\frac s2}|\xi^*-\xi|\right|=\frac{(2k-|\xi|)^{d-2}(1+|\xi^*|^2)^{\frac s2}2(k-|\xi|)}{|\xi|^{d-2}(k-|\xi|)(|\xi^*|+k)\sum_{j=0}^{n-1}|\xi^*|^{2j}k^{2(n-1-j)}}\lesssim k^{s-2n+1},
\end{align*}
we obtain 
\[
K_3+K_4\lesssim k^{s-2n+1}\|\phi\|_{H^{-s_1}(D)}\|\psi\|_{H^{-s_2}(B)}.
\]

We thus conclude that
\[
|\langle\mathcal H_k\phi,\psi\rangle|\lesssim k^{s-2n+1}\|\phi\|_{H^{-s_1}(D)}\|\psi\|_{H^{-s_2}(B)}\quad \forall~\phi\in C_0^\infty(D),~\psi\in C_0^\infty(B),
\]
which can be extended to $\phi\in H^{-s_1}(D)$ and $\psi\in H^{-s_2}(B)$ by a dense argument, completing the proof of \eqref{eq:H1}. 

The estimate in $L^\infty(B)$ can be derived following the same procedure as in \cite[Lemma 2.2]{LW24}. Indeed, utilizing the facts
\[
\widehat{\delta(x-\cdot)}(\xi)=e^{-{\rm i}x\cdot\xi}\widehat\delta(\xi)=e^{-{\rm i}x\cdot\xi},\quad \widehat{G(x,\cdot,k)}(\xi)=\frac{e^{-{\rm i}x\cdot\xi}}{|\xi|^{2n}-k^{2n}}
\]
and aforementioned estimates, we obtain 
\begin{align*}
\mathcal H_k\phi(x)&=\int_{\mathbb R^d}G(x,y,k)\phi(y)dy
=\int_{\mathbb R^d}(1+|\xi|^2)^{\frac s2}\widehat{G(x,\cdot,k)}(\xi)\widehat{\mathcal J^{-s}\phi}(\xi)d\xi\\
&=\int_{\mathbb R^d}\frac{(1+|\xi|^2)^{\frac{s+\frac{d+\epsilon}2}2}}{|\xi|^{2n}-k^{2n}}\widehat{\mathcal J^{-s}\phi}(\xi)\left(e^{-{\rm i}x\cdot\xi}(1+|\xi|^2)^{-\frac{d+\epsilon}4}\right)d\xi\\
&\lesssim k^{s-2n+1+\frac{d+\epsilon}2}\|\phi\|_{H^{-s}(D)}\|e^{-{\rm i}x\cdot(\cdot)}(1+|\cdot|^2)^{-\frac{d+\epsilon}4}\|_{H^1(\mathbb R^d)}\\
&\lesssim k^{s-2n+1+\frac{d+\epsilon}2}\|\phi\|_{H^{-s}(D)}
\end{align*}
uniformly for $x\in B$, thus completing the proof of \eqref{eq:H2}.
\end{proof}

\end{document}